\documentclass[11pt]{amsart}
\usepackage[margin=1.1in]{geometry}
\usepackage{amsmath,amssymb,amsthm,mathtools}
\usepackage{enumitem}
\usepackage{booktabs}
\usepackage{float}
\usepackage{array}
\usepackage{tikz}
\usetikzlibrary{arrows.meta,positioning,calc,decorations.pathmorphing,patterns}
\usepackage[colorlinks=true,linkcolor=blue!50!black,citecolor=green!40!black,urlcolor=blue!50!black]{hyperref}

\newtheorem{theorem}{Theorem}[section]
\newtheorem{proposition}[theorem]{Proposition}
\newtheorem{lemma}[theorem]{Lemma}
\newtheorem{corollary}[theorem]{Corollary}

\theoremstyle{definition}
\newtheorem{definition}[theorem]{Definition}
\newtheorem{example}[theorem]{Example}
\newtheorem{remark}[theorem]{Remark}

\newcommand{\Sp}{\operatorname{Sp}}
\newcommand{\Ins}{\operatorname{Ins}}
\newcommand{\Cm}{\operatorname{Cm}}
\newcommand{\Ury}{\operatorname{Ury}}
\newcommand{\Fun}{\operatorname{Fun}}
\newcommand{\cl}{\operatorname{cl}}
\newcommand{\Cl}{\operatorname{Cl}}
\newcommand{\Nins}{\mathcal{N}_{\mathrm{ins}}}
\newcommand{\Ncm}{\mathcal{N}_{\mathrm{cm}}}
\newcommand{\Dow}{\mathcal{D}}
\newcommand{\Homeo}{\operatorname{Homeo}}
\newcommand{\Aut}{\operatorname{Aut}}
\newcommand{\Ob}{X_{\mathrm{o}}}
\newcommand{\Cb}{X_{\mathrm{c}}}
\newcommand{\sq}{^{[2]}}
\newcommand{\sm}{\boxtimes}
\newcommand{\join}{\circledast}
\newcommand{\R}{\mathbb{R}}

\newcommand{\eqadj}[1]{\mathrel{\simeq_{#1}}}

\title[Graphs on topological spaces]{Graphs on topological spaces}
\author[P. J. Cameron and Vinod Kumar P. B.]{%
Peter J. Cameron$^{\,\mathrm{a}}$ and Vinod Kumar P. B.$^{\,\mathrm{b},*}$\\
$\vcenter{\hbox{\normalfont\footnotesize
\begin{tabular}{c}
\mbox{}\\
$^{\mathrm{a}}$\,School of Mathematics and Statistics, University of St Andrews, St Andrews, UK\\[2pt]
$^{\mathrm{b}}$\,Department of Mathematics, Muthoot Institute of Technology and Science, Kochi, Kerala, India
\end{tabular}}}$}
\thanks{$^{*}$Corresponding author.\newline
\emph{E-mail addresses:} \texttt{pjc20@st-andrews.ac.uk} (P. J. Cameron), \texttt{vinodkumarpb@mgits.ac.in} (Vinod Kumar P. B.).}


\begin{document}

\begin{abstract}
Graphs defined on groups (power graphs, enhanced power graphs, commuting graphs, generating graphs) form a hierarchy under inclusion of edge sets, and the comparison between members of the hierarchy has proved to be a rich source of problems. In this paper we develop an analogous theory for topological spaces. To a topological space $X$ we attach five graphs with vertex set $X$, all invariant under the homeomorphism group of $X$: the \emph{specialization graph} $\Sp(X)$, the \emph{inseparability graph} $\Ins(X)$, the \emph{closure-meeting graph} $\Cm(X)$, the \emph{Urysohn graph} $\Ury(X)$ and the \emph{functional graph} $\Fun(X)$. They form a diamond-shaped hierarchy whose members are null precisely when $X$ satisfies the separation axioms $T_1$, $T_2$, $T_{2\frac12}$ and ``functionally Hausdorff''. We prove that $\Ins(X)^{[2]}\subseteq \Ury(X)$ always, with equality for Alexandrov spaces; that the hierarchy is compatible with products via the strong product of graphs; and that every countable graph is the non-trivial part of the inseparability graph of a second countable $T_1$ space, whose Urysohn graph is then the square of the given graph and whose functional graph is its component relation. For finite spaces we characterize exactly which graphs arise as inseparability graphs (those in which every edge is covered by the closed neighbourhood of a simplicial vertex), determine all coincidences between members of the hierarchy, compute independence and clique cover numbers, and prove a homotopy theorem: via Dowker duality the complexes of inseparable and of closure-meeting families are always homotopy equivalent, and under a mild acyclicity hypothesis they recover the weak homotopy type of the space. The clique complex of the specialization graph always recovers the weak homotopy type, while the inseparability graph does not. Finally, following Cameron's work on topologies and filters preserved by permutation groups, we study the hierarchy for topologies invariant under a transitive group $G$: every graph in it is null or connected when $G$ is primitive, all five graphs coincide when the space is finite, primitivity and strong primitivity correspond to the vanishing of the Kolmogorov and specialization structures, and the neighbourhood filters of the five graphs are reverse-ordered along the hierarchy, which yields spaces whose specialization graph contains the Rado graph as a spanning subgraph. 
\end{abstract}

\maketitle

\section{Introduction}

In the last two decades there has been an explosion of research on graphs whose vertex set is an algebraic structure and whose edges reflect that structure: power graphs, enhanced power graphs, commuting graphs and generating graphs of groups; zero-divisor graphs of rings; subspace inclusion and subspace sum graphs of vector spaces; and many others. Cameron's survey~\cite{Cam-groups} emphasises a point of view which has been very productive: rather than studying one such graph in isolation, one should arrange several of them in a \emph{hierarchy} (for a finite group $G$, the power graph, the enhanced power graph, the commuting graph and the non-generating graph each being a spanning subgraph of the next) and ask when two members coincide, what their differences look like, which graphs occur as induced subgraphs, and how the graphs behave under products. Equalities in the hierarchy then characterize interesting classes of groups (EPPO groups, groups with cyclic or generalized quaternion Sylow subgroups, minimal non-abelian groups, Dedekind groups, $2$-Engel groups~\cite{Cam-groups,ACNS,Cam-AKCE}). The same philosophy has been applied to vector spaces~\cite{CDD} and, via ``super graphs'' built from equivalence relations, to a second dimension of the group hierarchy~\cite{ACNS}. In~\cite{Cam-AKCE} Cameron asks, more broadly, what graphs and algebraic structures can say to each other, and in~\cite{Cam-alg} he extends the theory from groups to arbitrary algebras in the sense of universal algebra: graphs whose definition only involves the subalgebra generated by a set of elements, such as the power graph, the enhanced power graph and the generating graph, make sense in any algebra, while others, such as the commuting graph, do not seem to generalize. Recent work on the cograph structure of enhanced power graphs~\cite{BFP} shows that the interplay between such graphs and classical graph classes remains an active topic.

The purpose of this paper is to carry this programme over to \emph{topological spaces}. At first sight a topology gives no binary operation from which to build a graph, but it does give a wealth of binary relations between points, namely the various ways in which two points may fail to be separated. The classical separation axioms $T_1$, $T_2$ (Hausdorff), $T_{2\frac12}$ (Urysohn) and ``functionally Hausdorff'' are statements that certain of these relations are trivial. We turn each relation into a graph. Concretely, for a topological space $X$ we define (Definition~\ref{def:graphs}) five graphs on the vertex set $X$:
\begin{itemize}[leftmargin=2em]
\item the \emph{specialization graph} $\Sp(X)$: $x\sim y$ if $x\in\cl\{y\}$ or $y\in\cl\{x\}$;
\item the \emph{inseparability graph} $\Ins(X)$: $x\sim y$ if $x,y$ have no disjoint neighbourhoods;
\item the \emph{closure-meeting graph} $\Cm(X)$: $x\sim y$ if $\cl\{x\}\cap\cl\{y\}\neq\emptyset$;
\item the \emph{Urysohn graph} $\Ury(X)$: $x\sim y$ if $x,y$ have no neighbourhoods with disjoint closures;
\item the \emph{functional graph} $\Fun(X)$: $x\sim y$ if $f(x)=f(y)$ for every continuous $f\colon X\to\R$.
\end{itemize}
These are the topological counterparts of, respectively, the power graph, the dual enhanced power graph, the enhanced power graph, the commuting graph and the non-generating graph. Indeed, for a finite space the specialization order plays the role of the ``is a power of'' preorder; $\Sp(X)$ is its comparability graph, exactly as the power graph is the comparability graph of the directed power graph. In fact, in the language of~\cite{Cam-alg}, where the power graph of an algebra joins $x$ and $y$ when $x\in\langle y\rangle$ or $y\in\langle x\rangle$, the specialization graph is obtained by replacing the subalgebra closure operator $\langle\,\cdot\,\rangle$ by the topological closure operator $\cl$; $\Cm(X)$ joins two points having a common upper bound, as the enhanced power graph joins two elements having a common ``root''; and $\Ins(X)$ joins two points having a common lower bound, as the dual enhanced power graph of~\cite[\S2.4]{Cam-groups} joins two elements whose cyclic subgroups meet non-trivially. Every homeomorphism of $X$ is an automorphism of each of the five graphs, which is the analogue of the requirement in~\cite{Cam-groups} that graphs on a group be invariant under $\Aut(G)$.

\subsection*{Summary of results}
Our main results are the following.

\smallskip\noindent\textbf{1. The separation hierarchy (Theorem~\ref{thm:hierarchy}).} For every space $X$,
\[
E(\Sp(X))\subseteq E(\Ins(X))\cap E(\Cm(X)),\qquad E(\Ins(X))\cup E(\Cm(X))\subseteq E(\Ury(X))\subseteq E(\Fun(X)),
\]
and $\Ins$ and $\Cm$ are in general incomparable (Figure~\ref{fig:diamond}). The graphs $\Sp,\Cm,\Ins,\Ury,\Fun$ are null if and only if $X$ is $T_1$, $T_1$, $T_2$, $T_{2\frac12}$, functionally Hausdorff, respectively. Moreover $\Fun(X)$ is a disjoint union of cliques, every point of $\Ins(X)$ at distance at most $2$ from another is joined to it in $\Ury(X)$, i.e.\ $\Ins(X)\sq\subseteq\Ury(X)$, and all five graphs of $X$ are obtained from those of the Kolmogorov quotient by blowing vertices up into cliques of twins. For Alexandrov spaces (in particular finite spaces) we have $\Ury(X)=\Ins(X)\sq$ and the classes of $\Fun(X)$ are the connected components; an explicit Hausdorff, non-Urysohn space shows that $\Ury\ne\Ins\sq$ in general (Example~\ref{ex:HnotU}).

\smallskip\noindent\textbf{2. Products and joins (Theorem~\ref{thm:products}).} For arbitrary spaces, $\Gamma(X\times Y)=\Gamma(X)\sm\Gamma(Y)$ for $\Gamma\in\{\Ins,\Cm,\Ury,\Fun\}$ (strong product), whereas $\Sp(X\times Y)=\Sp(X)\sm\Sp(Y)$ holds if and only if $X$ or $Y$ is $T_1$. We also compute the graphs of non-Hausdorff joins.

\smallskip\noindent\textbf{3. Universality (Theorem~\ref{thm:XGamma}).} For every countable graph $\Gamma$ we construct a second countable $T_1$ space $X_\Gamma$ such that $\Ins(X_\Gamma)$ is $\Gamma$ together with isolated vertices, $\Ury(X_\Gamma)$ is $\Gamma\sq$ together with isolated vertices, and the non-trivial classes of $\Fun(X_\Gamma)$ are the connected components of $\Gamma$. When $\Gamma$ is $2$-regular, $X_\Gamma$ is a non-Hausdorff $1$-manifold.

\smallskip\noindent\textbf{4. Realization for finite spaces (Theorem~\ref{thm:realization}).} A finite graph is isomorphic to $\Ins(X)$ for some finite space $X$ if and only if each of its edges lies in the closed neighbourhood of a simplicial vertex. Every $T_0$ realization of such a graph has exactly as many open points as there are classes of simplicial vertices, and this number equals both the independence number and the clique cover number. Consequently paths and cycles on at least four vertices, and trees of diameter at least three, are never inseparability graphs, although every finite graph is an induced subgraph of one. The smallest finite spaces with non-chordal and non-perfect inseparability graphs have $8$ and $10$ points (Proposition~\ref{prop:holes}).

\smallskip\noindent\textbf{5. Equality in the hierarchy (Theorem~\ref{thm:equality}).} For finite spaces we determine exactly when two members of the hierarchy coincide. For instance $\Ins(X)=\Ury(X)$ if and only if $\Ins(X)=\Fun(X)$, if and only if every connected component of $X$ is irreducible; $\Ins(X)=\Cm(X)$ if and only if each component contains exactly one open and exactly one closed point; and $\Sp(X)=\Ury(X)$ if and only if every component is a chain.

\smallskip\noindent\textbf{6. Homotopy (Theorem~\ref{thm:homotopy}).} Let $\Nins(X)$ and $\Ncm(X)$ be the simplicial complexes of finite sets of points which are ``jointly inseparable'' (their minimal open sets have a common point) and ``jointly closure-meeting'', respectively; their $1$-skeleta are $\Ins(X)$ and $\Cm(X)$. Using Dowker duality we prove that $\Nins(X)\simeq\Ncm(X)$ for every finite space $X$, even though $\Ins(X)$ and $\Cm(X)$ may be very different, and that under an acyclicity hypothesis on intersections of closures of open points (satisfied, e.g., by face posets of simplicial complexes) both are homotopy equivalent to the order complex, hence weakly equivalent to $X$ itself. By contrast the clique complex of $\Sp(X)$ is always the order complex of $X$, so that finite $T_0$ spaces with isomorphic specialization graphs are weakly homotopy equivalent (Corollary~\ref{cor:Sp-weak}); this fails for inseparability graphs (Example~\ref{ex:same-Ins}).

\smallskip\noindent\textbf{7. Group actions and filters (Theorems~\ref{thm:invariant} and~\ref{thm:filters}).} Cameron~\cite{Cam-top,CamV} related primitivity of a permutation group to separation properties of invariant topologies, and used the graph joining points with disjoint neighbourhoods (the complement of our $\Ins$) and neighbourhood filters of graphs, whose non-triviality characterizes graphs containing the Rado graph as a spanning subgraph. We show that for a topology invariant under a primitive group each of the five graphs is null or connected, that for finite spaces with transitive homeomorphism group all five graphs coincide, and that an invariant non-discrete Alexandrov $T_0$ topology has no open and no closed points. We prove that the neighbourhood filters of the five graphs are reverse-ordered along the hierarchy, that non-triviality of any of them forces all real functions to be constant, and that the comparability graph of a downward directed poset without minimal elements (for instance the generic poset) contains the Rado graph as a spanning subgraph.

\smallskip\noindent\textbf{8. Cographs and chordality (Proposition~\ref{prop:cograph}).} Motivated by the theorem of Bubboloni, Fumagalli and Praeger~\cite{BFP} that a cograph enhanced power graph is chordal, we prove that if $\Cm(X)$ is a cograph then it is chordal, for every space $X$, and that the same holds for $\Ins(X)$ when $X$ is Alexandrov but not in general. If $\Sp(X)$ coincides with $\Ins(X)$ or $\Cm(X)$, then it is trivially perfect.

\smallskip
The analogy with the group case is summarized in Table~\ref{tab:dictionary}. Section~\ref{sec:prelim} collects preliminaries. Section~\ref{sec:hierarchy} introduces the graphs and proves the hierarchy theorem; Section~\ref{sec:products} treats products and joins; Section~\ref{sec:universal} contains the universality construction. Sections~\ref{sec:finite}--\ref{sec:homotopy} concern finite spaces: invariants, realization, equality and homotopy. Section~\ref{sec:groups} treats topologies invariant under permutation groups and neighbourhood filters. Section~\ref{sec:data} reports computations.

\begin{table}[H]
\centering
\small
\begin{tabular}{@{}>{\raggedright\arraybackslash}p{0.46\textwidth}>{\raggedright\arraybackslash}p{0.46\textwidth}@{}}
\toprule
Graphs on a finite group $G$~\cite{Cam-groups} & Graphs on a topological space $X$ (this paper)\\
\midrule
directed power graph (preorder ``$y$ is a power of $x$'') & specialization preorder ($y\in\cl\{x\}$)\\
power graph (comparability graph) & specialization graph $\Sp(X)$\\
enhanced power graph (common root) & closure-meeting graph $\Cm(X)$ (common upper bound)\\
dual enhanced power graph ($\langle x\rangle\cap\langle y\rangle\neq 1$) & inseparability graph $\Ins(X)$ (common lower bound)\\
commuting graph & Urysohn graph $\Ury(X)$\\
non-generating graph & functional graph $\Fun(X)$\\
$\Aut(G)$-invariance & $\Homeo(X)$-invariance\\
equalities $\leftrightarrow$ EPPO, Dedekind, $2$-Engel groups & equalities $\leftrightarrow$ irreducible, chain components\\
power graphs are perfect & specialization graphs are perfect\\
power graph determines directed power graph~\cite{Cam-power2} & $\Sp(X)$ does \emph{not} determine $X$ (Remark~\ref{rem:Sp-not-determine})\\
Gruenberg--Kegel graph determined by the graphs & weak homotopy type determined by $\Sp(X)$\\
cograph enhanced power graphs are chordal~\cite{BFP} & cograph $\Cm(X)$ is chordal (Proposition~\ref{prop:cograph})\\
power graph of an algebra: $x\in\langle y\rangle$ or $y\in\langle x\rangle$~\cite{Cam-alg} & $\Sp(X)$: $x\in\cl\{y\}$ or $y\in\cl\{x\}$\\
permutation group preserving the graph & primitive $G$: invariant topologies are $T_0$~\cite{CamV}\\
Wielandt's strong primitivity & invariant topologies are $T_1$, i.e.\ $\Sp$ null~\cite{CamV}\\
neighbourhood filter, Rado graph~\cite{CamV} & filters $\mathcal F_{\Gamma(X)}$ reverse the hierarchy\\
\bottomrule
\end{tabular}
\caption{A dictionary between graphs on groups and graphs on spaces.}
\label{tab:dictionary}
\end{table}

\section{Preliminaries}\label{sec:prelim}

\subsection{Graphs}
All graphs are simple (loops are silently removed, as in~\cite{Cam-groups}) but may be infinite. For a graph $\Gamma$ we write $x\sim_\Gamma y$ (or $x\sim y$) for adjacency and
\[
x\eqadj{\Gamma} y \quad\Longleftrightarrow\quad x=y\ \text{ or }\ x\sim_\Gamma y .
\]
The \emph{closed neighbourhood} of $x$ is $N[x]=\{y: x\eqadj{\Gamma}y\}$ and the open neighbourhood is $N(x)=N[x]\setminus\{x\}$. A vertex $w$ is \emph{simplicial} if $N[w]$ is a clique. Two vertices are \emph{closed twins} if $N[x]=N[y]$. For a graph $\Gamma$ with metric $d_\Gamma$ (distance $\infty$ between components), the \emph{square} $\Gamma\sq$ has the same vertices, with $x\sim y$ iff $1\le d_\Gamma(x,y)\le 2$. The \emph{strong product} $\Gamma\sm\Delta$ has vertex set $V(\Gamma)\times V(\Delta)$ and $(x,y)\sim(x',y')$ iff $(x,y)\ne(x',y')$, $x\eqadj{\Gamma}x'$ and $y\eqadj{\Delta}y'$. We write $\Gamma\vee\Delta$ for the join (disjoint union plus all edges between the two parts), $K_S$ for the complete graph on a set $S$, and $\omega,\chi,\alpha,\theta$ for clique number, chromatic number, independence number and clique cover number. A graph is a \emph{cograph} if it has no induced $P_4$, \emph{chordal} if it has no induced cycle of length at least $4$, and \emph{perfect} if $\omega(\Delta)=\chi(\Delta)$ for every induced subgraph $\Delta$; by the Strong Perfect Graph Theorem~\cite{SPGT} this happens iff there is no induced odd cycle of length $\ge5$ or complement of one. The \emph{clique complex} $\Cl(\Gamma)$ is the simplicial complex whose simplices are the finite cliques of $\Gamma$.

\subsection{The specialization preorder}
Let $X$ be a topological space. For $x,y\in X$ we write
\begin{equation}\label{eq:spec}
x\le y \quad\Longleftrightarrow\quad \text{every open set containing $y$ contains $x$}\quad\Longleftrightarrow\quad y\in\cl\{x\}.
\end{equation}
This is the \emph{specialization preorder} (we follow the convention of Barmak~\cite{Barmak}, under which the closure of a point is its up-set). It is reflexive and transitive; it is antisymmetric iff $X$ is $T_0$, and it is the equality relation iff $X$ is $T_1$. Every continuous map is order preserving. Two points are \emph{topologically indistinguishable} if they have the same open neighbourhoods, i.e.\ $x\le y\le x$; the quotient $KX$ of $X$ by this equivalence is the \emph{Kolmogorov quotient}. The quotient map $q\colon X\to KX$ is open, closed and surjective, and every open or closed subset of $X$ is saturated for $q$; in particular $q^{-1}(\cl A)=\cl q^{-1}(A)$ for $A\subseteq KX$.

\subsection{Alexandrov and finite spaces}
A space is \emph{Alexandrov} if arbitrary intersections of open sets are open~\cite{Alexandroff}. Then every point $x$ has a smallest open neighbourhood
\[
U_x=\bigcap\{U: U\text{ open}, x\in U\}=\{y: y\le x\}={\downarrow}x ,
\]
the open sets are exactly the down-sets of $(X,\le)$ and the closed sets are exactly the up-sets. Because arbitrary unions of closed sets are closed, the closure of any subset is its up-set:
\begin{equation}\label{eq:closure}
\cl A={\uparrow}A=\bigcup_{a\in A}\cl\{a\},\qquad \cl\{x\}={\uparrow}x .
\end{equation}
In the terminology of Cameron~\cite{CamV}, Alexandrov topologies are the \emph{relational} topologies: those equal to the topology built from their own specialization preorder. This gives the classical equivalence between Alexandrov topologies and preorders, restricting to an equivalence between finite $T_0$ spaces and finite posets~\cite{Alexandroff,Stong,McCord,Barmak}. For a finite space we write $\Ob$ for the set of \emph{open points} ($\{x\}$ open, i.e.\ $x$ minimal) and $\Cb$ for the set of \emph{closed points} (maximal elements). For a finite $T_0$ space and $x\in X$ we put
\[
M(x)=U_x\cap \Ob,\qquad M^*(x)=\cl\{x\}\cap \Cb .
\]
Both are non-empty: $U_x$ is a finite non-empty poset, any minimal element of $U_x$ is minimal in $X$ because $U_x$ is a down-set, and dually. The \emph{height} of a finite $T_0$ space is the length of its longest chain. The \emph{opposite space} $X^{op}$ has the opposite order (open and closed sets interchanged). A finite $T_0$ space $X$ is \emph{irreducible} iff it has a generic point, i.e.\ a point $g$ with $\cl\{g\}=X$, iff it has a minimum.

\begin{remark}\label{rem:Sp-not-determine}
The first result proved about power graphs is Cameron's theorem~\cite{Cam-power2} that the power graph of a finite group determines the directed power graph up to isomorphism. In our dictionary (Table~\ref{tab:dictionary}) the directed power graph corresponds to the specialization preorder, which for a finite space carries the same information as the topology, and the power graph corresponds to the specialization graph $\Sp(X)$ of Definition~\ref{def:graphs}, the comparability graph of the specialization preorder. (Directed power graphs of arbitrary algebras are discussed in~\cite{Cam-alg}.) The analogue of Cameron's theorem is false. Since every finite poset is the specialization order of a finite $T_0$ space, every finite comparability graph is a specialization graph. But a poset is in general not determined by its comparability graph, even up to isomorphism and duality. Since $\Sp(X^{op})=\Sp(X)$ always, the best one could hope for is that $\Sp(X)$ determines $X$ up to homeomorphism and passage to the opposite space, and even this fails.

The smallest example has four points. Let $Y$ be the space with points $0,1,2,3$ ordered by $0<1<2$ and $1<3$; let $Y^{op}$ be its opposite; and let $D=\{0,1,2,3\}$ with $0<1<3$ and $0<2<3$ (the ``diamond'', which is homeomorphic to the product $\mathbb S\times\mathbb S$ of two Sierpi\'nski spaces). In each case exactly one pair of points is incomparable, so
\[
\Sp(Y)\cong\Sp(Y^{op})\cong\Sp(D)\cong K_4-e
\]
(Figure~\ref{fig:Sp-not-determine}). The three spaces are pairwise non-homeomorphic. $Y$ has one open point and two closed points, $Y^{op}$ has two open points and one closed point, and $D$ has one of each. Since $D^{op}\cong D$, none of the three is homeomorphic to the opposite of another. A direct check of the posets with at most three elements shows that four points is the minimum. Larger examples are easy to produce. For instance, if $V$ denotes the three-point space with one open point below two closed points, then $V\sqcup V$ and $V\sqcup V^{op}$ have isomorphic specialization graphs ($P_3\sqcup P_3$), but they are neither homeomorphic nor opposite to each other. In general the transitive orientations of a comparability graph are obtained by reversing implication classes independently, in the sense of Gallai~\cite{Gallai}, so a comparability graph typically has many non-isomorphic transitive orientations.

Two partial positive results remain. First, $\Sp(X)$ does determine the weak homotopy type of a finite $T_0$ space (Corollary~\ref{cor:Sp-weak}), as well as its height and width (Proposition~\ref{prop:Sp}). Accordingly, all three spaces above are contractible, since each has a minimum. Secondly, \emph{if $X$ is a connected finite $T_0$ space with at least two points and $\Sp(X)$ is triangle-free, then $\Sp(X)$ determines $X$ up to homeomorphism and passage to $X^{op}$.} Indeed, a chain of three points would give a triangle, so $X$ has height at most $1$. Since $X$ is connected with at least two points, no point is both open and closed, so $\Ob$ and $\Cb$ partition $X$ and form a bipartition of the connected bipartite graph $\Sp(X)$. Such a bipartition is unique up to interchanging the two parts. Once the part $\Ob$ is chosen, the order is recovered as $x<y$ iff $x\in\Ob$ and $x\sim y$. Choosing the other part gives $X^{op}$.
\end{remark}

\begin{figure}[ht]
\centering
\begin{tikzpicture}[scale=0.9,pt/.style={circle,fill,inner sep=1.5pt},every label/.style={font=\small}]
\begin{scope}
\node[pt,label=right:$0$] (0) at (0,0) {}; \node[pt,label=right:$1$] (1) at (0,1) {};
\node[pt,label=above:$2$] (2) at (-0.6,2) {}; \node[pt,label=above:$3$] (3) at (0.6,2) {};
\draw (0)--(1)--(2) (1)--(3);
\node at (0,-0.8) {$Y$};
\end{scope}
\begin{scope}[xshift=2.8cm]
\node[pt,label=below:$2$] (2) at (-0.6,0) {}; \node[pt,label=below:$3$] (3) at (0.6,0) {};
\node[pt,label=right:$1$] (1) at (0,1) {}; \node[pt,label=right:$0$] (0) at (0,2) {};
\draw (2)--(1)--(0) (3)--(1);
\node at (0,-0.8) {$Y^{op}$};
\end{scope}
\begin{scope}[xshift=5.6cm]
\node[pt,label=right:$0$] (0) at (0,0) {}; \node[pt,label=left:$1$] (1) at (-0.7,1) {};
\node[pt,label=right:$2$] (2) at (0.7,1) {}; \node[pt,label=right:$3$] (3) at (0,2) {};
\draw (0)--(1)--(3) (0)--(2)--(3);
\node at (0,-0.8) {$D\cong\mathbb S\times\mathbb S$};
\end{scope}
\draw[-{Stealth},thick,gray] (7.0,1)--(8.2,1);
\begin{scope}[xshift=9.6cm]
\node[pt,label=left:$a$] (a) at (-0.8,1) {}; \node[pt,label=right:$b$] (b) at (0.8,1) {};
\node[pt,label=below:$c$] (c) at (0,0) {}; \node[pt,label=above:$d$] (d) at (0,2) {};
\draw[thick,blue!60!black] (a)--(c)--(b)--(d)--(a) (c)--(d);
\node at (0,-0.8) {$K_4-e$};
\end{scope}
\end{tikzpicture}
\caption{Remark~\ref{rem:Sp-not-determine}: three pairwise non-homeomorphic $4$-point spaces (Hasse diagrams; open sets are down-sets), none the opposite of another, with the same specialization graph $K_4-e$. The incomparable pair ($2,3$ in $Y$ and $Y^{op}$, $1,2$ in $D$) becomes the non-edge $ab$.}
\label{fig:Sp-not-determine}
\end{figure}
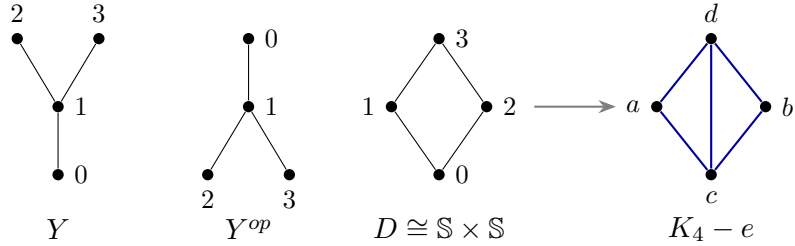

\subsection{Separation axioms}
Recall that $X$ is $T_1$ if points are closed; $T_2$ (Hausdorff) if distinct points have disjoint neighbourhoods; $T_{2\frac12}$ (Urysohn) if distinct points have neighbourhoods with disjoint closures; and \emph{functionally Hausdorff} (or completely Hausdorff) if for distinct $x,y$ there is a continuous $f\colon X\to[0,1]$ with $f(x)=0$, $f(y)=1$. Each property implies the preceding one and none of the converses holds; see~\cite{Engelking,SteenSeebach}.

\subsection{Complexes, nerves and Dowker duality}
For a finite poset $P$ the \emph{order complex} $\Delta(P)$ is the simplicial complex of non-empty chains. McCord~\cite{McCord} proved that for every finite $T_0$ space $X$ there is a weak homotopy equivalence $|\Delta(X)|\to X$. If $P$ has a maximum or minimum then $\Delta(P)$ is a cone and so is contractible. We shall use the nerve lemma in the following form.

\begin{lemma}[Nerve lemma~\cite{Borsuk,Bjorner}]\label{lem:nerve}
Let $K$ be a finite simplicial complex and $(K_i)_{i\in I}$ a finite family of subcomplexes with $\bigcup_i K_i=K$. Suppose that for every non-empty $J\subseteq I$ the intersection $\bigcap_{i\in J}K_i$ is either empty or contractible. Then $|K|$ is homotopy equivalent to the nerve $\mathcal N=\{J\ne\emptyset:\bigcap_{i\in J}K_i\neq\emptyset\}$.
\end{lemma}

For a relation $R\subseteq A\times B$ between finite sets, the two \emph{Dowker complexes} are
\begin{gather*}
K_R=\{\sigma\subseteq A \text{ finite non-empty}: \exists b\ \forall a\in\sigma\ aRb\},\\
L_R=\{\tau\subseteq B\text{ finite non-empty}: \exists a\ \forall b\in\tau\ aRb\}.
\end{gather*}

\begin{lemma}[~\cite{Dowker}]\label{lem:dowker}
$|K_R|\simeq|L_R|$ for every relation $R$ between finite sets.
\end{lemma}

\begin{proof}
We deduce this from Lemma~\ref{lem:nerve}. We may discard the elements of $A$ and $B$ that are related to nothing, since they are vertices of neither complex. For $b\in B$ let $R^{-1}(b)=\{a: aRb\}\neq\emptyset$ and let $F_b$ be the full simplex on $R^{-1}(b)$. By definition $K_R=\bigcup_bF_b$. For $\emptyset\ne J\subseteq B$ we have $\bigcap_{b\in J}F_b$ equal to the full simplex on $\bigcap_{b\in J}R^{-1}(b)$, which is contractible when non-empty. The nerve of the cover $(F_b)$ consists of the sets $J$ with $\bigcap_{b\in J}R^{-1}(b)\ne\emptyset$, that is, of the sets $J\subseteq B$ for which some $a$ is related to every element of $J$; this is precisely $L_R$. Lemma~\ref{lem:nerve} gives $|K_R|\simeq|L_R|$.
\end{proof}

\section{Five graphs on a topological space and the separation hierarchy}\label{sec:hierarchy}

\begin{definition}\label{def:graphs}
Let $X$ be a topological space. The following five graphs all have vertex set $X$; in each case $x\ne y$.
\begin{enumerate}[label=(\roman*)]
\item The \emph{specialization graph} $\Sp(X)$: $x\sim y$ iff $x\le y$ or $y\le x$, i.e.\ $y\in\cl\{x\}$ or $x\in\cl\{y\}$.
\item The \emph{inseparability graph} $\Ins(X)$: $x\sim y$ iff $V\cap W\ne\emptyset$ for all neighbourhoods $V$ of $x$ and $W$ of $y$.
\item The \emph{closure-meeting graph} $\Cm(X)$: $x\sim y$ iff $\cl\{x\}\cap\cl\{y\}\neq\emptyset$.
\item The \emph{Urysohn graph} $\Ury(X)$: $x\sim y$ iff $\cl V\cap\cl W\ne\emptyset$ for all neighbourhoods $V$ of $x$ and $W$ of $y$.
\item The \emph{functional graph} $\Fun(X)$: $x\sim y$ iff $f(x)=f(y)$ for all continuous $f\colon X\to\R$.
\end{enumerate}
\end{definition}

For a graph type $\Gamma$ we sometimes write $x\sim_\Gamma y$ instead of $x\sim_{\Gamma(X)}y$. Each graph is defined purely in terms of the topology, so every homeomorphism $X\to X$ is an automorphism of each of them: $\Homeo(X)\le\Aut(\Gamma(X))$. As for groups, the automorphism group of the graph is usually much larger. The complement of $\Ins(X)$, in which two points are joined when they have disjoint open neighbourhoods, is the \emph{separation graph} $\operatorname{Sep}(X)$; it appears in Cameron's treatment of topologies invariant under primitive groups~\cite{Cam-top,CamV}, and we return to it in Section~\ref{sec:groups}.

\begin{figure}[ht]
\centering
\begin{tikzpicture}[node distance=1.25cm and 1.6cm, box/.style={draw,rounded corners,inner sep=4pt,minimum width=1.3cm}]
\node[box] (N) {null};
\node[box,above=0.9cm of N] (Sp) {$\Sp$};
\node[box,above left=0.9cm and 0.8cm of Sp] (Ins) {$\Ins$};
\node[box,above right=0.9cm and 0.8cm of Sp] (Cm) {$\Cm$};
\node[box,above right=0.9cm and 0.8cm of Ins] (Ury) {$\Ury$};
\node[box,above=0.9cm of Ury] (Fun) {$\Fun$};
\node[box,above=0.9cm of Fun] (K) {complete};
\draw[-{Stealth}] (N)--(Sp); \draw[-{Stealth}] (Sp)--(Ins); \draw[-{Stealth}] (Sp)--(Cm);
\draw[-{Stealth}] (Ins)--(Ury); \draw[-{Stealth}] (Cm)--(Ury); \draw[-{Stealth}] (Ury)--(Fun); \draw[-{Stealth}] (Fun)--(K);
\node[right=0.35cm of Sp,font=\footnotesize,text width=3cm] {null $\iff T_1$};
\node[left=0.35cm of Ins,font=\footnotesize,align=right] {null $\iff T_2$};
\node[right=0.35cm of Cm,font=\footnotesize] {null $\iff T_1$};
\node[right=0.35cm of Ury,font=\footnotesize] {null $\iff T_{2\frac12}$};
\node[right=0.35cm of Fun,font=\footnotesize] {null $\iff$ functionally Hausdorff};
\draw[dashed,gray] (Ins) to[bend left=10] node[midway,above,font=\scriptsize,black] {$\Ins\sq$} (Ury);
\end{tikzpicture}
\caption{The separation hierarchy of Theorem~\ref{thm:hierarchy}. Arrows denote inclusion of edge sets. $\Ins$ and $\Cm$ are incomparable in general, $\Ins\sq\subseteq \Ury$ always, with equality for Alexandrov spaces.}
\label{fig:diamond}
\end{figure}
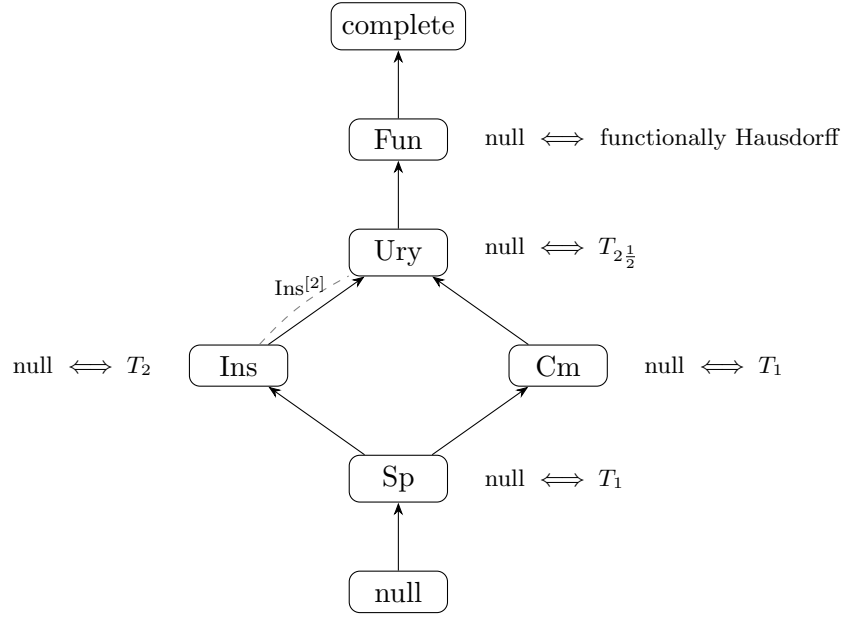

\vspace{1cm}
\begin{theorem}\label{thm:hierarchy}
Let $X$ be any topological space.
\begin{enumerate}[label=\textup{(\alph*)}]
\item $E(\Sp(X))\subseteq E(\Ins(X))\cap E(\Cm(X))$ and $E(\Ins(X))\cup E(\Cm(X))\subseteq E(\Ury(X))\subseteq E(\Fun(X))$. In general $E(\Ins(X))$ and $E(\Cm(X))$ are incomparable.
\item $\Sp(X)$ is null $\iff$ $\Cm(X)$ is null $\iff$ $X$ is $T_1$; $\Ins(X)$ is null $\iff$ $X$ is Hausdorff; $\Ury(X)$ is null $\iff$ $X$ is Urysohn; $\Fun(X)$ is null $\iff$ $X$ is functionally Hausdorff.
\item $\Fun(X)$ is a disjoint union of complete graphs; its classes are closed subsets of $X$. In particular $E(\Fun(X))$ contains the transitive closure of each of $E(\Sp(X))$, $E(\Ins(X))$, $E(\Cm(X))$, $E(\Ury(X))$.
\item $E(\Ins(X)\sq)\subseteq E(\Ury(X))$: if $x$ and $y$ are at distance at most $2$ in $\Ins(X)$ then they are adjacent in $\Ury(X)$.
\item \textup{(Kolmogorov reduction)} Let $q\colon X\to KX$ be the Kolmogorov quotient and $\Gamma$ one of the five graph types. For $x\ne y$ in $X$, $x\sim_{\Gamma(X)}y$ iff $q(x)\eqadj{\Gamma(KX)}q(y)$. Hence $\Gamma(X)$ is obtained from $\Gamma(KX)$ by replacing each vertex $c$ by a clique on $q^{-1}(c)$ whose vertices are closed twins.
\item \textup{(Alexandrov spaces)} If $X$ is Alexandrov, then $x\sim_{\Ins}y$ iff $U_x\cap U_y\neq\emptyset$; $x\sim_{\Cm}y$ iff ${\uparrow}x\cap{\uparrow}y\neq\emptyset$; $\Ury(X)=\Ins(X)\sq$; and the classes of $\Fun(X)$ are the connected components of $X$, which coincide with the connected components of $\Sp(X)$.
\end{enumerate}
\end{theorem}

\begin{proof}
Throughout, ``neighbourhood'' may be replaced by ``open neighbourhood'' in Definition~\ref{def:graphs}(ii) and (iv), since every neighbourhood contains an open one and shrinking neighbourhoods (and hence their closures) only makes the non-emptiness conditions harder to satisfy.

\emph{(a)} Let $x\sim_{\Sp}y$, say $x\le y$, i.e.\ every open set containing $y$ contains $x$. If $V$ and $W$ are open neighbourhoods of $x$ and $y$, then $x\in W$, so $x\in V\cap W\neq\emptyset$; hence $x\sim_{\Ins}y$. Also $y\in\cl\{x\}$ and $y\in\cl\{y\}$, so $\cl\{x\}\cap\cl\{y\}\ne\emptyset$ and $x\sim_{\Cm}y$. This proves $E(\Sp)\subseteq E(\Ins)\cap E(\Cm)$.

If $x\sim_{\Ins}y$ and $V,W$ are neighbourhoods of $x,y$, then $\emptyset\ne V\cap W\subseteq\cl V\cap\cl W$, so $x\sim_{\Ury}y$. If $x\sim_{\Cm}y$, choose $z\in\cl\{x\}\cap\cl\{y\}$; for neighbourhoods $V\ni x$ and $W\ni y$ we have $\cl\{x\}\subseteq \cl V$ and $\cl\{y\}\subseteq\cl W$, so $z\in\cl V\cap\cl W$ and $x\sim_{\Ury}y$.

Now let $x\sim_{\Ury}y$ and suppose, for a contradiction, that some continuous $f\colon X\to\R$ has $f(x)\neq f(y)$. Composing with an affine map of $\R$ and then with the retraction $t\mapsto\min(1,\max(0,t))$ of $\R$ onto $[0,1]$, we may assume $f(x)=0$ and $f(y)=1$. Put $V=f^{-1}([0,\frac13))$ and $W=f^{-1}((\frac23,1])$, open neighbourhoods of $x$ and $y$. By continuity $\cl V\subseteq f^{-1}([0,\frac13])$ and $\cl W\subseteq f^{-1}([\frac23,1])$, so $\cl V\cap \cl W=\emptyset$, a contradiction. Hence $x\sim_{\Fun}y$.

For incomparability, let $S=\{a,b,c,d\}$ be the four-point space whose open sets are generated by $\{a\},\{b\},\{a,b,c\},\{a,b,d\}$ (the \emph{pseudo-circle}~\cite{McCord}; see Figure~\ref{fig:pseudocircle}). Here $a,b<c,d$. The points $a$ and $b$ are separated by the open sets $\{a\},\{b\}$, but $c\in\cl\{a\}\cap\cl\{b\}$, so $ab\in E(\Cm)\setminus E(\Ins)$. The points $c,d$ are closed and distinct, so $cd\notin E(\Cm)$, but every neighbourhood of either contains $\{a,b\}$, so $cd\in E(\Ins)\setminus E(\Cm)$.

\emph{(b)} By~\eqref{eq:spec}, $\Sp(X)$ is null iff the specialization preorder is equality iff $X$ is $T_1$. If $X$ is $T_1$ then $\cl\{x\}=\{x\}$ for every $x$, so $\Cm(X)$ is null; conversely if $\Cm(X)$ is null then by (a) so is $\Sp(X)$, and $X$ is $T_1$. The statements for $\Ins,\Ury,\Fun$ are just the definitions of the separation axioms, rewritten as the absence of edges (for $\Fun$ we use the reduction to $f(x)=0,f(y)=1$ made in the proof of (a)).

\emph{(c)} The relation $x\equiv y\iff f(x)=f(y)$ for all $f\in C(X,\R)$ is an equivalence relation, and $\Fun(X)$ is the graph of this relation with loops removed, hence a disjoint union of cliques on its classes. The class of $x$ is $\bigcap_{f}f^{-1}(f(x))$, an intersection of closed sets. Since $\equiv$ is transitive and contains all edges of $\Ury(X)$ by (a), and hence those of $\Sp,\Ins,\Cm$, it contains their transitive closures.

\emph{(d)} Let $x\sim_{\Ins}z\sim_{\Ins}y$ with $x\ne y$ (the case $x\sim_{\Ins}y$ is covered by (a)). Let $V,W$ be open neighbourhoods of $x$ and $y$. Every neighbourhood $O$ of $z$ meets $V$ (as $x\sim_{\Ins}z$), so $z\in\cl V$; similarly $z\in\cl W$. Thus $\cl V\cap\cl W\ni z$ and $x\sim_{\Ury}y$.

\emph{(e)} Recall that the open sets of $X$ are exactly the sets $q^{-1}(O)$ with $O$ open in $KX$, and likewise for closed sets. Let $x\neq y$.

If $q(x)=q(y)$, then $x$ and $y$ have the same open neighbourhoods; so $x\le y\le x$ and $x\sim_{\Sp}y$, whence $x$ and $y$ are adjacent in all five graphs by (a). Moreover, for every graph type the conditions defining adjacency of $x$ with a third vertex $z$ only involve the open neighbourhoods of $x$ (for $\Sp$, $\Ins$, $\Ury$), the set $\cl\{x\}$ (for $\Cm$; note $\cl\{x\}=\cl\{y\}$ because $x\le y\le x$), or the values $f(x)$ with $f$ continuous (for $\Fun$; note $x\sim_{\Fun}y$). Hence $x,y$ are closed twins.

Now suppose $q(x)\ne q(y)$. For $\Sp$: $y\in\cl\{x\}$ iff every open set containing $y$ contains $x$ iff every open subset of $KX$ containing $q(y)$ contains $q(x)$ (opens of $X$ are saturated). For $\Ins$: open neighbourhoods of $x$ are the $q^{-1}(O)$ with $q(x)\in O$ and $q^{-1}(O)\cap q^{-1}(O')=q^{-1}(O\cap O')$, which is empty iff $O\cap O'=\emptyset$ because $q$ is surjective. For $\Cm$: $\cl\{x\}=q^{-1}(\cl\{q(x)\})$ since closed sets are saturated, and preimages under a surjection intersect iff the sets intersect. For $\Ury$: $\cl q^{-1}(O)=q^{-1}(\cl O)$, and we conclude as for $\Ins$. For $\Fun$: every continuous $f\colon X\to\R$ is constant on the fibres of $q$ (if $x\le x'\le x$ then $f(x)\le f(x')\le f(x)$ in the specialization preorder of $\R$, which is equality since $\R$ is $T_1$), so $f=\bar f\circ q$ with $\bar f$ continuous because $q$ is a quotient map; conversely $\bar f\circ q$ is continuous for every continuous $\bar f$. Thus in all five cases adjacency of $x,y$ in $\Gamma(X)$ is equivalent to adjacency of $q(x),q(y)$ in $\Gamma(KX)$.

\emph{(f)} Let $X$ be Alexandrov. Since $U_x$ is the smallest neighbourhood of $x$, $x\sim_{\Ins}y$ iff $U_x\cap U_y\ne\emptyset$, i.e.\ iff $x$ and $y$ have a common lower bound. By~\eqref{eq:closure}, $x\sim_{\Cm}y$ iff ${\uparrow}x\cap{\uparrow}y\ne\emptyset$, i.e.\ iff they have a common upper bound.

For the Urysohn graph, again the smallest neighbourhoods give the strongest condition, so $x\sim_{\Ury}y$ iff $\cl U_x\cap\cl U_y\ne\emptyset$. By~\eqref{eq:closure}, $\cl U_x={\uparrow}{\downarrow}x$. Hence $x\sim_{\Ury}y$ iff there are $a\le x$, $b\le y$ and $w$ with $a\le w\ge b$. Given such a zig-zag $x\ge a\le w\ge b\le y$, the point $a$ lies in $U_x\cap U_w$ and $b\in U_w\cap U_y$, so either $x=w$ or $x\sim_{\Ins}w$, and either $w=y$ or $w\sim_{\Ins}y$; hence $d_{\Ins}(x,y)\le 2$. This shows $E(\Ury(X))\subseteq E(\Ins(X)\sq)$, and the reverse inclusion is (d).

Finally, let $\approx$ be the equivalence relation generated by $\le$ (equivalently, $x\approx y$ iff $x,y$ lie in the same connected component of $\Sp(X)$). Let $C$ be an $\approx$-class. If $x\in C$ then $U_x={\downarrow}x\subseteq C$ and $\cl\{x\}={\uparrow}x\subseteq C$; hence $C$ is a union of open sets and a union of closed sets, and since unions of closed sets are closed in an Alexandrov space, $C$ is clopen. $C$ is connected: if $A\subseteq C$ is non-empty and clopen in $C$ (hence in $X$) and $x\in A$, then $A\supseteq U_x\cup\cl\{x\}$, so $A$ is closed under taking comparable elements, whence $A=C$. Therefore the $\approx$-classes are exactly the connected components of $X$. The characteristic function of a component $C$ is continuous (as $C$ is clopen), so points in different components are not $\Fun$-adjacent. Conversely, if $f\colon X\to\R$ is continuous and $x\le y$, then $y\in\cl\{x\}$ gives $f(y)\in\cl\{f(x)\}=\{f(x)\}$; so $f$ is constant on $\approx$-classes and any two points of one component are $\Fun$-adjacent.
\end{proof}

\begin{figure}[ht]
\centering
\begin{tikzpicture}[scale=0.85,every node/.style={font=\small},pt/.style={circle,fill,inner sep=1.4pt}]
\begin{scope}
\node[pt,label=below:$a$] (a) at (0,0) {}; \node[pt,label=below:$b$] (b) at (1.4,0) {};
\node[pt,label=above:$c$] (c) at (0,1.4) {}; \node[pt,label=above:$d$] (d) at (1.4,1.4) {};
\draw (a)--(c) (a)--(d) (b)--(c) (b)--(d);
\node at (0.7,-1) {$S$ (Hasse)};
\end{scope}
\foreach \name/\x/\edges in {{$\Sp(S)$}/3.3/{a/c,a/d,b/c,b/d},{$\Ins(S)$}/6.0/{a/c,a/d,b/c,b/d,c/d},{$\Cm(S)$}/8.7/{a/c,a/d,b/c,b/d,a/b},{$\Ury(S)=\Fun(S)$}/11.4/{a/c,a/d,b/c,b/d,a/b,c/d}}{
\begin{scope}[xshift=\x cm]
\node[pt,label=below:$a$] (a) at (0,0) {}; \node[pt,label=below:$b$] (b) at (1.4,0) {};
\node[pt,label=above:$c$] (c) at (0,1.4) {}; \node[pt,label=above:$d$] (d) at (1.4,1.4) {};
\foreach \u/\v in \edges {\draw[thick,blue!60!black] (\u)--(\v);}
\node at (0.7,-1) {\name};
\end{scope}}
\end{tikzpicture}
\caption{The pseudo-circle $S$ (weakly homotopy equivalent to $S^1$) and its graphs. $\Ins(S)$ and $\Cm(S)$ are both $K_4$ minus an edge, but different edges are missing.}
\label{fig:pseudocircle}
\end{figure}
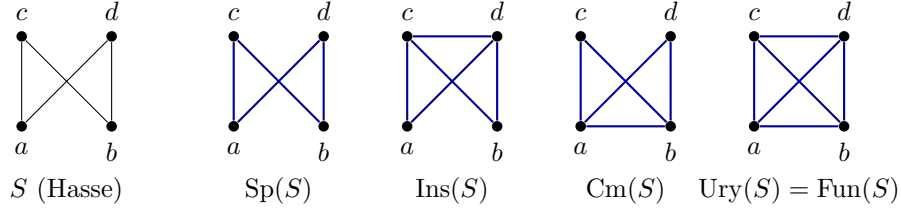

\begin{example}[A Hausdorff space with $\Ury\neq\Ins\sq$]\label{ex:HnotU}
Let
\[
X=\{a,b\}\cup\{c_n:n\ge1\}\cup\{p_{n,k},\,q_{n,k}:n,k\ge1\}.
\]
Declare every $p_{n,k}$ and $q_{n,k}$ isolated, and take as basic neighbourhoods
\begin{align*}
B_K(c_n)&=\{c_n\}\cup\{p_{n,k},q_{n,k}:k\ge K\},\\
B_N(a)&=\{a\}\cup\{p_{n,k}:n\ge N,\ k\ge1\},\qquad B_N(b)=\{b\}\cup\{q_{n,k}:n\ge N,\ k\ge1\}.
\end{align*}
(See Figure~\ref{fig:HnotU}.) It is routine that these sets form a neighbourhood base of a topology. The space is Hausdorff: $B_1(a)\cap B_1(b)=\emptyset$; $B_{n+1}(a)\cap B_1(c_n)=\emptyset$ and likewise for $b$; and $B_1(c_n)\cap B_1(c_m)=\emptyset$ for $n\ne m$. Thus $\Ins(X)$ is null and so is $\Ins(X)\sq$. On the other hand every neighbourhood of $c_n$ meets $B_N(a)$ when $n\ge N$, so
\[
\cl B_N(a)=B_N(a)\cup\{c_n:n\ge N\},\qquad \cl B_N(b)=B_N(b)\cup\{c_n:n\ge N\},
\]
and $\cl B_N(a)\cap\cl B_{N'}(b)\ni c_{\max(N,N')}$. Hence $a\sim_{\Ury}b$, and $\Ury(X)$ has exactly one edge, $ab$. So the inclusion of Theorem~\ref{thm:hierarchy}(d) can be strict, and Theorem~\ref{thm:hierarchy}(f) fails without the Alexandrov hypothesis. Similarly, since there are Urysohn spaces which are not functionally Hausdorff~\cite{SteenSeebach}, $\Fun(X)$ may be non-null while $\Ury(X)$ is null; so the transitive closure in Theorem~\ref{thm:hierarchy}(c) can also be strictly contained in $E(\Fun(X))$.
\end{example}

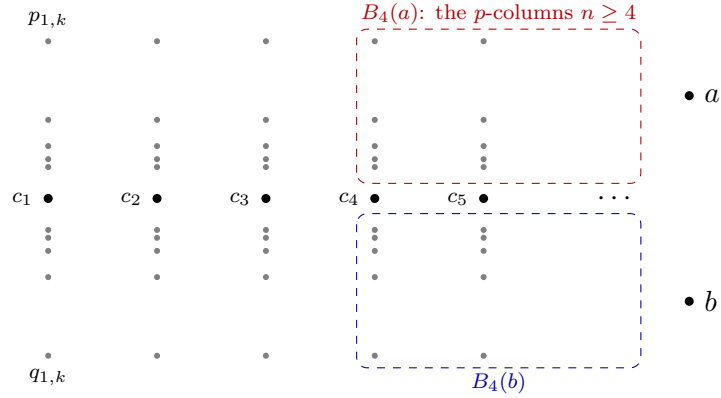
\begin{figure}[ht]
\centering
\begin{tikzpicture}[scale=0.8,pt/.style={circle,fill,inner sep=1.2pt},sp/.style={circle,fill=gray,inner sep=0.8pt}]
\foreach \n in {1,2,3,4,5}{
  \node[pt,label={[font=\scriptsize]left:$c_{\n}$}] (c\n) at (1.8*\n,0) {};
  \foreach \k in {1,2,3,4,5}{
    \node[sp] at (1.8*\n,{2.6/\k}) {};
    \node[sp] at (1.8*\n,{-2.6/\k}) {};
  }
}
\node[font=\scriptsize] at (1.8,2.95) {$p_{1,k}$};
\node[font=\scriptsize] at (1.8,-2.95) {$q_{1,k}$};
\node at (11.2,0) {$\cdots$};
\node[pt,label=right:$a$] (a) at (12.4,1.7) {};
\node[pt,label=right:$b$] (b) at (12.4,-1.7) {};
\draw[dashed,red!70!black,rounded corners] (6.9,0.25) rectangle (11.6,2.8);
\node[red!70!black,font=\scriptsize] at (9.25,3.05) {$B_4(a)$: the $p$-columns $n\ge 4$};
\draw[dashed,blue!70!black,rounded corners] (6.9,-0.25) rectangle (11.6,-2.8);
\node[blue!70!black,font=\scriptsize] at (9.25,-3.05) {$B_4(b)$};
\end{tikzpicture}
\caption{The space of Example~\ref{ex:HnotU}. Column $n$ accumulates at $c_n$ from above and below; the upper halves of the columns accumulate at $a$, the lower halves at $b$. The closures of $B_N(a)$ and $B_N(b)$ both contain the tail $\{c_n:n\ge N\}$.}
\label{fig:HnotU}
\end{figure}

\section{Products, sums and joins}\label{sec:products}

The next result is the topological analogue of~\cite[Proposition~6.1]{Cam-groups}, where the commuting graph of a direct product of groups is the strong product of the commuting graphs, while for the power graph the analogous statement fails.

\begin{theorem}\label{thm:products}
Let $X$ and $Y$ be topological spaces, and give $X\times Y$ the product topology.
\begin{enumerate}[label=\textup{(\alph*)}]
\item For $\Gamma\in\{\Ins,\Cm,\Ury,\Fun\}$ we have $\Gamma(X\times Y)=\Gamma(X)\sm\Gamma(Y)$.
\item $E(\Sp(X\times Y))\subseteq E(\Sp(X)\sm\Sp(Y))$, with equality if and only if $X$ or $Y$ is $T_1$.
\item For the topological sum, $\Gamma(X\sqcup Y)=\Gamma(X)\sqcup\Gamma(Y)$ for all five graph types.
\item Let $X,Y\ne\emptyset$ and let $X\join Y$ be the non-Hausdorff join, the set $X\sqcup Y$ whose open sets are the open sets of $X$ together with the sets $X\cup V$ for $V$ open in $Y$. Then
\[
\Sp(X\join Y)=\Sp(X)\vee\Sp(Y),\quad \Ins(X\join Y)=\Ins(X)\vee K_Y,\quad \Cm(X\join Y)=K_X\vee\Cm(Y),
\]
and $\Ury(X\join Y)=\Fun(X\join Y)$ is complete.
\end{enumerate}
\end{theorem}

\begin{proof}
Write $\pi_X,\pi_Y$ for the projections. Recall that the sets $A\times B$ with $A,B$ open form a base, that $\cl(A\times B)=\cl A\times\cl B$, and in particular $\cl\{(x,y)\}=\cl\{x\}\times\cl\{y\}$.

\emph{(a), $\Ins$.} We show that $(x,y)$ and $(x',y')$ can be separated by disjoint open sets iff $x,x'$ can be (in $X$) or $y,y'$ can be (in $Y$); note that equal points cannot be separated, which is why the strong product, with its ``equal or adjacent'' clause, appears. If $A\ni x$, $A'\ni x'$ are disjoint and open, then $A\times Y$ and $A'\times Y$ separate the two pairs; similarly in the second coordinate. Conversely, suppose $O\ni(x,y)$ and $O'\ni(x',y')$ are disjoint and open. Choose basic sets $A\times B\subseteq O$ and $A'\times B'\subseteq O'$ containing the respective points. Then $(A\cap A')\times(B\cap B')=(A\times B)\cap(A'\times B')=\emptyset$, so $A\cap A'=\emptyset$ or $B\cap B'=\emptyset$, and $x,x'$ or $y,y'$ are separated. Taking negations, $(x,y)\eqadj{\Ins}(x',y')$ iff $x\eqadj{\Ins}x'$ and $y\eqadj{\Ins}y'$, which is the definition of $\Ins(X)\sm\Ins(Y)$.

\emph{(a), $\Ury$.} The argument is identical once we note that, if $O\supseteq A\times B$ and $O'\supseteq A'\times B'$ have disjoint closures, then $\cl(A\times B)\cap\cl(A'\times B')=(\cl A\cap\cl A')\times(\cl B\cap\cl B')=\emptyset$, so $\cl A\cap\cl A'=\emptyset$ or $\cl B\cap\cl B'=\emptyset$; and conversely $\cl(A\times Y)\cap\cl(A'\times Y)=(\cl A\cap\cl A')\times Y$.

\emph{(a), $\Cm$.} $\cl\{(x,y)\}\cap\cl\{(x',y')\}=(\cl\{x\}\cap\cl\{x'\})\times(\cl\{y\}\cap\cl\{y'\})$, which is non-empty iff both factors are; and $\cl\{x\}\cap\cl\{x'\}\ne\emptyset$ iff $x\eqadj{\Cm}x'$.

\emph{(a), $\Fun$.} If $f(x)\ne f(x')$ for some continuous $f$ on $X$, then $f\circ\pi_X$ separates $(x,y)$ from $(x',y')$; similarly for $Y$. Conversely, suppose that $x\eqadj{\Fun}x'$ and $y\eqadj{\Fun}y'$ and let $g\colon X\times Y\to\R$ be continuous. The map $g(\,\cdot\,,y)\colon X\to\R$ is continuous, so $g(x,y)=g(x',y)$; the map $g(x',\cdot\,)\colon Y\to\R$ is continuous, so $g(x',y)=g(x',y')$. Hence $g(x,y)=g(x',y')$.

\emph{(b)} Since $\cl\{(x,y)\}=\cl\{x\}\times\cl\{y\}$, the specialization preorder of $X\times Y$ is the product preorder: $(x,y)\le(x',y')$ iff $x\le x'$ and $y\le y'$. If $(x,y)$ and $(x',y')$ are comparable, say $(x,y)\le(x',y')$, then $x,x'$ are equal or comparable and so are $y,y'$; this gives the inclusion. If $X$ is $T_1$ and $(x,y)\sim(x',y')$ in $\Sp(X)\sm\Sp(Y)$, then $x=x'$ (as $\Sp(X)$ is null) and $y\sim_{\Sp}y'$, so the pairs are comparable; similarly if $Y$ is $T_1$. Conversely, if neither is $T_1$ choose $x<x'$ in $X$ and $y<y'$ in $Y$ (strict meaning $x\le x'$, $x'\not\le x$). Then $(x,y')$ and $(x',y)$ are adjacent in $\Sp(X)\sm\Sp(Y)$, but $(x,y')\le(x',y)$ would force $y'\le y$ and $(x',y)\le(x,y')$ would force $x'\le x$; so they are not adjacent in $\Sp(X\times Y)$.

\emph{(c)} In $X\sqcup Y$ both summands are clopen; the neighbourhoods, closures and point closures of points of $X$ are computed in $X$, points of $X$ are separated from points of $Y$ by the open sets $X$ and $Y$ (whose closures are $X$ and $Y$), and by the characteristic function of $X$. Continuous real functions on $X\sqcup Y$ are pairs of continuous functions.

\emph{(d)} Every open set of $X\join Y$ containing a point $y\in Y$ contains $X$; hence $x\le y$ for all $x\in X$, $y\in Y$, and the specialization preorder restricted to $X$ (respectively $Y$) is that of $X$ (respectively $Y$), because $X$ is an open subspace and the traces on $Y$ of the open sets containing points of $Y$ are exactly the open sets of $Y$. This gives the formula for $\Sp$. Every pair of neighbourhoods of two points of $Y$ contains $X\neq\emptyset$, so $Y$ is a clique in $\Ins$; each $x\in X$ is $\Sp$-adjacent and hence $\Ins$-adjacent to each $y\in Y$; and $x,x'\in X$ are separated in $X\join Y$ iff they are separated in the open subspace $X$. This gives $\Ins(X\join Y)=\Ins(X)\vee K_Y$. The closed sets of $X\join Y$ are the sets $(X\setminus U)\cup Y$ ($U$ open in $X$) and $Y\setminus V$ ($V$ open in $Y$). Hence $\cl\{y\}=\cl_Y\{y\}$ for $y\in Y$ and $\cl\{x\}=\cl_X\{x\}\cup Y$ for $x\in X$. So any two points of $X$ are $\Cm$-adjacent (their closures contain $Y\ne\emptyset$), each $x\in X$ is $\Cm$-adjacent to each $y\in Y$, and $y,y'\in Y$ are $\Cm$-adjacent iff they are in $\Cm(Y)$. Finally $E(\Ury)\supseteq E(\Ins)\cup E(\Cm)$, which is already complete.
\end{proof}
\begin{remark}\label{rem:join-commute}
While non-Hausdorff join is not commutative (in this respect it resembles the wreath product of groups), the corresponding graphs $\Sp$, $\Ury$ and $\Fun$  do commute.
\end{remark}
\begin{corollary}\label{cor:product-separation}
For non-empty spaces, $X\times Y$ is Hausdorff (respectively Urysohn, functionally Hausdorff, $T_1$) iff both $X$ and $Y$ are; and the non-Hausdorff join of two non-empty spaces is never $T_1$.
\end{corollary}
\begin{proof}
A strong product $\Gamma\sm\Delta$ of graphs on non-empty vertex sets is null iff both factors are null. Apply Theorem~\ref{thm:hierarchy}(b) and Theorem~\ref{thm:products}.
\end{proof}

This recovers classical facts from a uniform graph-theoretic statement. We also note that, by Theorem~\ref{thm:products}(a), questions about $\Ins$, $\Cm$, $\Ury$ and $\Fun$ for finite products reduce to the factors, exactly as questions about commuting graphs of nilpotent groups reduce to Sylow subgroups~\cite[\S6]{Cam-groups}.

\section{Universality and the spaces $X_\Gamma$}\label{sec:universal}

For the power graph of a group, the finite graphs that occur as induced subgraphs are exactly the comparability graphs, whereas every finite graph embeds in an enhanced power graph or a commuting graph~\cite[\S5]{Cam-groups}. The same dichotomy holds here. Since specialization graphs are comparability graphs of the specialization preorder (which, after Kolmogorov reduction, is a partial order), and since every partial order is the specialization order of an Alexandrov topology, we obtain:

\begin{proposition}\label{prop:induced}
\begin{enumerate}[label=\textup{(\alph*)}]
\item A finite graph is an induced subgraph of $\Sp(X)$ for some \textup{(}finite\textup{)} space $X$ if and only if it is a comparability graph.
\item Every finite graph $\Gamma$ is an induced subgraph of $\Ins(X)$, of $\Cm(X)$ and of $\Ury(X)$ for suitable finite $T_0$ spaces $X$ of height at most $1$.
\end{enumerate}
\end{proposition}

\begin{proof}
(a) An induced subgraph of a comparability graph is a comparability graph; conversely a comparability graph of a poset $P$ is $\Sp$ of the Alexandrov space of $P$.

(b) Let $\Gamma=(V,E)$. Let $P$ be the poset on $V\sqcup E$ with $e<v$ whenever $v\in e$. In the corresponding finite space, $u,v\in V$ have a common lower bound iff $uv\in E$, so $\Ins(P)$ induces $\Gamma$ on $V$ by Theorem~\ref{thm:hierarchy}(f). For $\Cm$ and $\Ury$ use the opposite poset $P^{op}$ (with $v<e$ for $v\in e$): now the elements of $V$ are minimal, two of them have a common upper bound iff they are adjacent, so $\Cm(P^{op})$ induces $\Gamma$ on $V$; and for minimal $u,v$ the zig-zag condition in the proof of Theorem~\ref{thm:hierarchy}(f) reduces to $u\le w\ge v$, so $\Ury(P^{op})$ also induces $\Gamma$ on $V$.
\end{proof}

In fact much more is true: we can prescribe the whole inseparability graph, not just an induced subgraph, provided we allow infinite spaces and ignore isolated vertices; and the construction simultaneously controls the Urysohn and functional graphs.

\begin{definition}\label{def:XGamma}
Let $\Gamma=(V,E)$ be a countable graph. For each edge $e\in E$ let $L_e=(0,1)\times\{e\}$ (a ``ray''), and set
\[
X_\Gamma=V\ \sqcup\ \bigcup_{e\in E}L_e .
\]
Topologize $X_\Gamma$ by the family $\mathcal B$ consisting of
\begin{itemize}[leftmargin=2em]
\item the sets $I\times\{e\}$ with $I\subseteq(0,1)$ an open interval and $e\in E$; and
\item for $v\in V$ and $0<\varepsilon\le 1$, the \emph{stars} $S_\varepsilon(v)=\{v\}\cup\bigcup_{e\ni v}(0,\varepsilon)\times\{e\}$.
\end{itemize}
\end{definition}

Thus each ray $L_e$, $e=\{u,v\}$, is an open interval whose end $0$ is approached by \emph{both} $u$ and $v$; the ends at $1$ are left open. An isolated vertex $v$ has $S_\varepsilon(v)=\{v\}$. See Figure~\ref{fig:XGamma}.

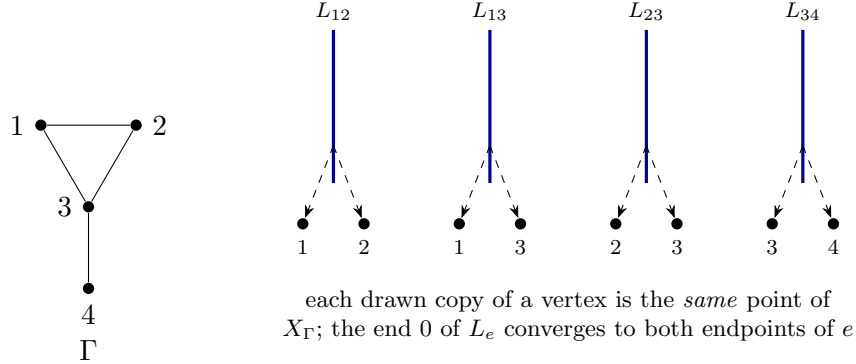
\begin{figure}[ht]
\centering
\begin{tikzpicture}[scale=0.9,pt/.style={circle,fill,inner sep=1.5pt}]
\begin{scope}
\node[pt,label=left:$1$] (1) at (0,1.2) {}; \node[pt,label=right:$2$] (2) at (1.4,1.2) {};
\node[pt,label=left:$3$] (3) at (0.7,0) {}; \node[pt,label=below:$4$] (4) at (0.7,-1.2) {};
\draw (1)--(2)--(3)--(1) (3)--(4);
\node at (0.7,-2.1) {$\Gamma$};
\end{scope}
\begin{scope}[xshift=4.3cm]
\foreach \lab/\x/\u/\v in {12/0/1/2,13/2.3/1/3,23/4.6/2/3,34/6.9/3/4}{
 \draw[very thick,blue!60!black] (\x,0.35) -- (\x,2.6);
 \node[font=\scriptsize,above] at (\x,2.6) {$L_{\lab}$};
 \node[pt,label={[font=\scriptsize]below:$\u$}] (A\lab) at (\x-0.45,-0.25) {};
 \node[pt,label={[font=\scriptsize]below:$\v$}] (B\lab) at (\x+0.45,-0.25) {};
 \draw[dashed,-{Stealth}] (\x,0.9)--(A\lab);
 \draw[dashed,-{Stealth}] (\x,0.9)--(B\lab);
}
\node[font=\footnotesize,text width=8.2cm,align=center] at (3.45,-1.6) {each drawn copy of a vertex is the \emph{same} point of $X_\Gamma$; the end $0$ of $L_e$ converges to both endpoints of $e$};
\end{scope}
\end{tikzpicture}
\caption{The space $X_\Gamma$ for the ``paw'' $\Gamma$. Here $\Ins(X_\Gamma)$ restricted to $\{1,2,3,4\}$ is $\Gamma$, $\Ury(X_\Gamma)$ restricted to it is $\Gamma\sq=K_4$, and all ray points are isolated in every graph.}
\label{fig:XGamma}
\end{figure}

\begin{theorem}\label{thm:XGamma}
Let $\Gamma=(V,E)$ be a countable graph. Then $\mathcal B$ is a base of a topology on $X_\Gamma$, and:
\begin{enumerate}[label=\textup{(\alph*)}]
\item $X_\Gamma$ is a second countable $T_1$ space; in particular $\Sp(X_\Gamma)$ and $\Cm(X_\Gamma)$ are null.
\item $\Ins(X_\Gamma)=\Gamma\ \sqcup\ \overline{K}_{\,X_\Gamma\setminus V}$: two vertices of $\Gamma$ are inseparable iff they are adjacent in $\Gamma$, and every point of a ray is isolated.
\item $\Ury(X_\Gamma)=\Gamma\sq\ \sqcup\ \overline{K}_{\,X_\Gamma\setminus V}$.
\item The non-singleton classes of $\Fun(X_\Gamma)$ are exactly the vertex sets of the connected components of $\Gamma$ with at least two vertices.
\item If every vertex of $\Gamma$ has degree $2$, then $X_\Gamma$ is locally homeomorphic to $\R$, i.e.\ a second countable, non-Hausdorff $1$-manifold \textup{(}provided $E\neq\emptyset$\textup{)}.
\end{enumerate}
\end{theorem}

\begin{proof}
\emph{$\mathcal B$ is a base.} Every point lies in some member of $\mathcal B$. We check that the intersection of two members is a union of members. Two interval sets intersect in an interval set or in the empty set. An interval set $I\times\{e\}$ meets a star $S_\varepsilon(v)$ in $(I\cap(0,\varepsilon))\times\{e\}$ if $v\in e$ and in $\emptyset$ otherwise; both are interval sets or empty. For stars at the same vertex, $S_\varepsilon(v)\cap S_\delta(v)=S_{\min(\varepsilon,\delta)}(v)$. For stars at distinct vertices $u\neq v$, a ray $L_e$ contributes to both only if $e\ni u$ and $e\ni v$, i.e.\ $e=\{u,v\}$; hence
\begin{equation}\label{eq:starint}
S_\varepsilon(u)\cap S_\delta(v)=\begin{cases}(0,\min(\varepsilon,\delta))\times\{uv\}&\text{if }uv\in E,\\ \emptyset&\text{otherwise}.\end{cases}
\end{equation}
This is an interval set or empty. Hence $\mathcal B$ is a base. The neighbourhoods of a ray point $(t,e)$ are those of $t$ in $(0,1)$, transported to $L_e$ (each ray is an open subset homeomorphic to $(0,1)$), and the stars $S_\varepsilon(v)$, $\varepsilon\in(0,1]$, form a neighbourhood base at $v$.

\emph{(a)} Since $V$ and $E$ are countable, the interval sets with rational endpoints and the stars with rational $\varepsilon$ form a countable base (every $S_\varepsilon(v)$ contains $S_{\varepsilon'}(v)$ for rational $\varepsilon'<\varepsilon$, and every interval set is a union of rational ones). For $T_1$ we show that every point $p$ has a neighbourhood base avoiding any given other point $p'$. If $p=(t,e)$ is a ray point, a small interval around $t$ in $L_e$ avoids $p'$. If $p=v\in V$, then $S_\varepsilon(v)$ contains no vertex other than $v$, and it avoids a ray point $(t,e)$ as soon as $\varepsilon\le t$. Hence $X_\Gamma$ is $T_1$, and $\Sp,\Cm$ are null by Theorem~\ref{thm:hierarchy}(b).

\emph{(b)} By~\eqref{eq:starint}, for distinct $u,v\in V$ the basic neighbourhoods $S_\varepsilon(u)$ and $S_\delta(v)$ always meet if $uv\in E$, and are disjoint (already for $\varepsilon=\delta=1$) if $uv\notin E$. Since stars form neighbourhood bases, $u\sim_{\Ins}v$ iff $uv\in E$. Next let $p=(t,e)$ be a ray point and $p'\ne p$. If $p'=(t',e)$ lies on the same ray, disjoint intervals around $t$ and $t'$ separate them. If $p'=(t',e')$ with $e'\ne e$, the rays themselves are disjoint open sets. If $p'=v\in V$, then the interval $(t/2,1)\times\{e\}$ and the star $S_{t/2}(v)$ are disjoint. Hence ray points are isolated in $\Ins(X_\Gamma)$.

\emph{(c)} We first compute the closure of a star. Let $v\in V$ and $0<\varepsilon<1$. We claim
\begin{equation}\label{eq:starclosure}
\cl S_\varepsilon(v)=\{v\}\cup N_\Gamma(v)\cup\bigcup_{e\ni v}(0,\varepsilon]\times\{e\}.
\end{equation}
Indeed, a vertex $w\in N_\Gamma(v)$ belongs to the closure because each $S_\delta(w)$ meets $S_\varepsilon(v)$ by~\eqref{eq:starint}; a vertex $w\notin N_\Gamma[v]$ does not, because $S_1(w)\cap S_\varepsilon(v)=\emptyset$. A ray point $(t,e)$ with $e\ni v$ is in the closure iff $t\le\varepsilon$ (closure of $(0,\varepsilon)$ in $(0,1)$), and a ray point on a ray $L_e$ with $v\notin e$ is not, since $L_e$ is an open set disjoint from $S_\varepsilon(v)$. This proves~\eqref{eq:starclosure}.

Let $u\ne v$ be vertices. Since stars form neighbourhood bases and shrinking a neighbourhood shrinks its closure, $u\sim_{\Ury}v$ iff $\cl S_\varepsilon(u)\cap\cl S_\delta(v)\neq\emptyset$ for all small $\varepsilon,\delta$. By~\eqref{eq:starclosure}, the vertices in this intersection are exactly those in $N_\Gamma[u]\cap N_\Gamma[v]$, and the ray points in it lie on a common ray $L_e$ with $e\ni u,v$, which only happens if $uv\in E$, in which case $N_\Gamma[u]\cap N_\Gamma[v]\ni u$ anyway. Therefore $u\sim_{\Ury}v$ iff $N_\Gamma[u]\cap N_\Gamma[v]\neq\emptyset$ iff $d_\Gamma(u,v)\le 2$, i.e.\ iff $uv\in E(\Gamma\sq)$.

It remains to see that ray points are isolated in $\Ury(X_\Gamma)$. Let $p=(t,e)$. For a second ray point $p'$ on a ray $L_{e'}$, choose disjoint closed intervals $[t-\eta,t+\eta]$ and $[t'-\eta,t'+\eta]$ in $(0,1)$ (or arbitrary ones if $e'\ne e$): the open intervals they contain are neighbourhoods whose closures in $X_\Gamma$ are these closed intervals. Indeed, a set $J\times\{e\}$ with $J\subseteq(0,1)$ a compact interval is closed in $X_\Gamma$: every vertex $w$ has a star $S_\delta(w)$ with $\delta<\min J$ disjoint from it, and every ray point outside it has an interval neighbourhood disjoint from it. For $p'=v\in V$, take $\eta<t/2$; by~\eqref{eq:starclosure} the closure of $S_{t/2}(v)$ meets $L_e$ at most in $(0,t/2]\times\{e\}$, which is disjoint from $[t-\eta,t+\eta]\times\{e\}$.

\emph{(d)} By Theorem~\ref{thm:hierarchy}(c), each class of $\Fun(X_\Gamma)$ contains the vertex set of a component of $\Gamma$, since $\Fun$ contains the transitive closure of $\Ins$ and by (b) $\Ins$ restricted to $V$ is $\Gamma$. We now construct enough continuous functions to show that there are no further identifications.

(i) Let $C$ be the vertex set of a connected component of $\Gamma$ and put $A_C=C\cup\bigcup\{L_e: e\subseteq C\}$. $A_C$ is open, being the union of the stars $S_1(v)$ ($v\in C$) and of the rays $L_e$ ($e\subseteq C$), since every edge at a vertex of $C$ lies inside $C$. Its complement is the analogous set for $V\setminus C$, also open. Hence the characteristic function $\mathbf 1_{A_C}$ is continuous, and it separates the vertices of $C$ from all vertices outside $C$.

(ii) For $e_0\in E$ define $g_{e_0}\colon X_\Gamma\to\R$ by $g_{e_0}(t,e_0)=t$ and $g_{e_0}=0$ elsewhere (on all vertices and all other rays). It is continuous on the open set $L_{e_0}$ and on the open set $\bigcup_{e\ne e_0}L_e$; at a vertex $v$ we have $g_{e_0}(S_\varepsilon(v))\subseteq[0,\varepsilon)$, so $g_{e_0}$ is continuous at $v$ as well. The function $g_{e_0}$ takes the value $t>0$ at $(t,e_0)$ and a different value at every other point of $X_\Gamma$.

By (ii), every ray point forms a singleton class of $\Fun$; by (i), vertices in different components of $\Gamma$ lie in different classes. This proves (d). In particular an isolated vertex of $\Gamma$ is a singleton class.

\emph{(e)} If $\deg v=2$ with edges $e,e'$, then $S_1(v)=((0,1)\times\{e\})\cup\{v\}\cup((0,1)\times\{e'\})$, and the map sending $(t,e)\mapsto -t$, $v\mapsto0$, $(t,e')\mapsto t$ is a homeomorphism $S_1(v)\to(-1,1)$ (it maps the basic sets of $S_1(v)$, namely stars $S_\varepsilon(v)$ and intervals, onto the intervals $(-\varepsilon,\varepsilon)$ and onto open intervals avoiding $0$, which form a base of $(-1,1)$). Ray points have neighbourhoods homeomorphic to open intervals. Hence $X_\Gamma$ is a $1$-manifold; it is second countable by (a) and non-Hausdorff by (b) whenever $E\ne\emptyset$.
\end{proof}

\begin{corollary}\label{cor:defect}
Call the subgraph of $\Ins(X)$ induced on its non-isolated vertices the \emph{Hausdorff defect graph} of $X$. Every countable graph without isolated vertices is the Hausdorff defect graph of a second countable $T_1$ space. In particular the finite cycles are Hausdorff defect graphs of second countable non-Hausdorff $1$-manifolds.
\end{corollary}

\begin{remark}
Theorem~\ref{thm:XGamma} should be contrasted with the finite case treated in Section~\ref{sec:realization}, where only a restricted class of graphs occurs as inseparability graphs. The difference is that in $X_\Gamma$ the ``witnesses'' of inseparability are germs of rays rather than points: in a finite space two inseparable points always have a common point in their smallest neighbourhoods, and this point is simplicial in $\Ins(X)$.
\end{remark}

\section{Finite spaces: invariants of the hierarchy}\label{sec:finite}

In this section $X$ is a finite space. By Theorem~\ref{thm:hierarchy}(e) all five graphs of $X$ are obtained from those of the finite $T_0$ space $KX$ by blowing up vertices into cliques of closed twins; so we lose little by assuming that $X$ is $T_0$, i.e.\ a finite poset with the Alexandrov topology in which open sets are down-sets. By Theorem~\ref{thm:hierarchy}(f) the five graphs are described order-theoretically as follows ($x\ne y$):
\begin{align*}
x\sim_{\Sp}y&\iff x\le y\text{ or }y\le x, & x\sim_{\Ins}y&\iff M(x)\cap M(y)\ne\emptyset,\\
x\sim_{\Cm}y&\iff M^*(x)\cap M^*(y)\neq\emptyset, & x\sim_{\Ury}y&\iff d_{\Ins}(x,y)\le 2,\\
x\sim_{\Fun}y&\iff x,y\text{ in the same component}.
\end{align*}
(For $\Ins$: a common lower bound $z$ of $x,y$ lies above a minimal element, which is then in $M(x)\cap M(y)$; dually for $\Cm$.) In particular $\Cm(X)=\Ins(X^{op})$. The following is immediate from these descriptions and is the analogue of the ``zig-zag'' description of the Urysohn graph.

\begin{lemma}\label{lem:Ury-open}
Let $X$ be a finite $T_0$ space and $x\ne y$. Then $x\sim_{\Ury}y$ iff there exist $m\in M(x)$ and $m'\in M(y)$ with $\cl\{m\}\cap\cl\{m'\}\neq\emptyset$ \textup{(}possibly $m=m'$\textup{)}.
\end{lemma}
\begin{proof}
By the proof of Theorem~\ref{thm:hierarchy}(f), $x\sim_{\Ury}y$ iff there is a zig-zag $x\ge a\le w\ge b\le y$. Replacing $a$ and $b$ by minimal elements below them preserves the zig-zag.
\end{proof}

\begin{proposition}\label{prop:Sp}
Let $X$ be a finite $T_0$ space.
\begin{enumerate}[label=\textup{(\alph*)}]
\item $\Sp(X)$ is the comparability graph of $(X,\le)$; hence it is perfect, $\omega(\Sp(X))=\chi(\Sp(X))$ is the number of points of a longest chain, and $\alpha(\Sp(X))=\theta(\Sp(X))$ is the width of $X$.
\item $\Cl(\Sp(X))=\Delta(X)$; hence $|\Cl(\Sp(X))|$ is weakly homotopy equivalent to $X$.
\item $\Sp(X)$ is a cograph iff $X$ can be built from one-point spaces by topological sums and non-Hausdorff joins.
\end{enumerate}
\end{proposition}
\begin{proof}
(a) A proper colouring of a comparability graph is a partition into antichains, so $\chi=\omega$ is Mirsky's theorem~\cite{Mirsky}; since induced subgraphs of comparability graphs are comparability graphs, $\Sp(X)$ is perfect; $\alpha=\theta=$ width is Dilworth's theorem~\cite{Dilworth}. (b) A finite set of points is a clique in the comparability graph iff it is a chain; now apply McCord's theorem~\cite{McCord}. (c) By Theorem~\ref{thm:products}(c),(d), $\Sp(X\sqcup Y)=\Sp(X)\sqcup\Sp(Y)$ and $\Sp(X\join Y)=\Sp(X)\vee\Sp(Y)$, and cographs are exactly the graphs built from $K_1$ by disjoint unions and joins~\cite{CLS}. Conversely, a poset whose comparability graph is a cograph is series-parallel~\cite{Jung,VTL}, i.e.\ built by disjoint sums and ordinal sums; and the ordinal sum $P\oplus Q$ is the non-Hausdorff join of the associated spaces.
\end{proof}

\begin{corollary}\label{cor:Sp-weak}
If $X$ and $Y$ are finite $T_0$ spaces with $\Sp(X)\cong\Sp(Y)$, then $X$ and $Y$ are weakly homotopy equivalent.
\end{corollary}
\begin{proof}
A graph isomorphism induces an isomorphism of clique complexes, hence $\Delta(X)\cong\Delta(Y)$ by Proposition~\ref{prop:Sp}(b), and $X\simeq_w|\Delta(X)|\cong|\Delta(Y)|\simeq_w Y$.
\end{proof}

This is the topological analogue of~\cite[Theorem~2.9]{Cam-groups}, which says that groups with isomorphic power graphs have the same Gruenberg--Kegel graph. We shall see in Example~\ref{ex:same-Ins} that the analogue fails for the inseparability graph. On the other hand, Remark~\ref{rem:Sp-not-determine} shows that, unlike the power graph, $\Sp(X)$ does not determine $X$ even up to homeomorphism and duality.

\begin{proposition}\label{prop:alpha}
Let $X$ be a finite $T_0$ space. Then
\[
\alpha(\Ins(X))=\theta(\Ins(X))=|\Ob|,\quad \alpha(\Cm(X))=\theta(\Cm(X))=|\Cb|,
\]
both $\alpha(\Fun(X))$ and $\theta(\Fun(X))$ equal the number of connected components of $X$, and $\alpha(\Ury(X))$ is the largest number of open points with pairwise disjoint closures. Moreover $\omega(\Ins(X))\ge\max_{m\in\Ob}|\cl\{m\}|$.
\end{proposition}
\begin{proof}
Distinct open points $m,m'$ satisfy $M(m)=\{m\}\ne\{m'\}=M(m')$, so $\Ob$ is independent in $\Ins(X)$ and $\alpha\ge|\Ob|$. On the other hand, for $m\in\Ob$ the set $\cl\{m\}={\uparrow}m$ is a clique of $\Ins(X)$ (its elements share the lower bound $m$), and these cliques cover $X$ since every $x$ lies above an open point. Hence $\theta\le|\Ob|$. Since $\alpha\le\theta$ for every graph, all three numbers agree; the bound on $\omega$ also follows. The statement for $\Cm$ follows by applying this to $X^{op}$, and that for $\Fun$ is clear. Finally let $I$ be independent in $\Ury(X)$ and choose $m_x\in M(x)$ for $x\in I$. By Lemma~\ref{lem:Ury-open} the closures $\cl\{m_x\}$ ($x\in I$) are pairwise disjoint, in particular the $m_x$ are distinct; so $I$ gives a set of $|I|$ open points with pairwise disjoint closures. Conversely such a set of open points is independent in $\Ury(X)$ by Lemma~\ref{lem:Ury-open}, since $M(m)=\{m\}$.
\end{proof}

\begin{proposition}\label{prop:distances}
Let $X$ be a finite space and $x,y$ in the same component. Then
\[
d_{\Ins}(x,y)\le d_{\Sp}(x,y)\le 2\,d_{\Ins}(x,y),\qquad d_{\Cm}(x,y)\le d_{\Sp}(x,y)\le 2\,d_{\Cm}(x,y),\]
\[ d_{\Ury}(x,y)=\lceil d_{\Ins}(x,y)/2\rceil .
\]
In particular the connected components of $\Sp(X)$, $\Ins(X)$, $\Cm(X)$ and $\Ury(X)$ are the connected components of $X$.
\end{proposition}
\begin{proof}
The first inequality in each chain follows from $E(\Sp)\subseteq E(\Ins)\cap E(\Cm)$. If $u\sim_{\Ins}v$ with common lower bound $z$, then $u\ge z\le v$ is a path of length at most $2$ in $\Sp(X)$; dually for $\Cm$. The formula for $d_{\Ury}$ follows from $\Ury(X)=\Ins(X)\sq$ (Theorem~\ref{thm:hierarchy}(f)), since in the square of a graph distances are halved and rounded up. The last statement follows from the inequalities and Theorem~\ref{thm:hierarchy}(f).
\end{proof}

\section{Which graphs are inseparability graphs of finite spaces?}\label{sec:realization}

Proposition~\ref{prop:induced} shows that every finite graph is an \emph{induced subgraph} of the inseparability graph of a finite space. The question of which graphs are \emph{equal} to such a graph is more delicate, and its answer is our next main result.

\begin{definition}
A finite graph $\Gamma$ is \emph{simplicially covered} if for every edge $uv$ there is a simplicial vertex $w$ with $u,v\in N[w]$ (i.e.\ every edge lies in a clique of the form $N[w]$, $w$ simplicial). Two simplicial vertices $w,w'$ are \emph{s-equivalent} if $w\eqadj{\Gamma}w'$; we write $\sigma(\Gamma)$ for the number of s-equivalence classes.
\end{definition}

\begin{lemma}\label{lem:s-equiv}
If $w,w'$ are adjacent simplicial vertices then $N[w]=N[w']$. Consequently s-equivalence is an equivalence relation on the set of simplicial vertices, and distinct classes are non-adjacent.
\end{lemma}
\begin{proof}
Since $w'\in N[w]$ and $N[w]$ is a clique, every vertex of $N[w]$ is equal or adjacent to $w'$, so $N[w]\subseteq N[w']$; by symmetry equality holds. If $w\eqadj{}w'\eqadj{}w''$ are simplicial then $w''\in N[w']=N[w]$, which proves transitivity, and shows that a vertex adjacent to a member of a class lies in the common closed neighbourhood of the class.
\end{proof}

\begin{theorem}\label{thm:realization}
For a finite graph $\Gamma=(V,E)$ the following are equivalent.
\begin{enumerate}[label=\textup{(\roman*)}]
\item $\Gamma\cong\Ins(X)$ for some finite topological space $X$.
\item $\Gamma$ is simplicially covered.
\item $\Gamma$ is isomorphic to the intersection graph of a family $(A_v)_{v\in V}$ of non-empty subsets of a finite set $S$ such that every singleton $\{s\}$, $s\in S$, is a member of the family.
\item $\Gamma\cong\Ins(X)$ for a finite $T_0$ space $X$ of height at most $1$ with $|X|=|V|$.
\end{enumerate}
Moreover, if these conditions hold, then for every finite $T_0$ space $X$ and every isomorphism $\varphi\colon\Ins(X)\to\Gamma$, the set $\varphi(\Ob)$ contains exactly one vertex from each s-equivalence class and no other vertex. In particular
\[
|\Ob|=\sigma(\Gamma)=\alpha(\Gamma)=\theta(\Gamma)
\]
is an invariant of $\Gamma$, and $X$ has at least $\sigma(\Gamma)$ open points and at most $|V|-\sigma(\Gamma)$ non-open points. The same statements hold for $\Cm$ in place of $\Ins$, with open points replaced by closed points.
\end{theorem}

\begin{proof}
\emph{(i)$\Rightarrow$(ii).} Let $X$ be a finite space with $\Ins(X)\cong\Gamma$; we identify $\Gamma$ with $\Ins(X)$. Here we do not assume $T_0$. Call a point $m$ \emph{basic} if $U_m$ is minimal (with respect to inclusion) among the sets $U_x$, $x\in X$. Every $U_z$ contains $U_m$ for some basic $m$, because $X$ is finite.

We claim that for a basic point $m$,
\begin{equation}\label{eq:Nm}
N[m]=\{y\in X: m\in U_y\}=\{y: U_m\subseteq U_y\}.
\end{equation}
Indeed, if $m\in U_y$ then $U_m\subseteq U_y$ (as $U_y$ is an open set containing $m$), so $m\in U_m\cap U_y$ and $y\eqadj{\Ins}m$. Conversely, if $y\eqadj{\Ins}m$ then there is $z\in U_m\cap U_y$; then $U_z\subseteq U_m$, and minimality of $U_m$ forces $U_z=U_m$, so $m\in U_m=U_z\subseteq U_y$. This proves~\eqref{eq:Nm}.

It follows that every basic point is simplicial in $\Ins(X)$: any two points of $N[m]$ contain $m$ in their minimal open sets and are therefore equal or adjacent. Now let $xy$ be an edge. There is $z\in U_x\cap U_y$, and a basic $m$ with $U_m\subseteq U_z$. Then $m\in U_x\cap U_y$, so $x,y\in N[m]$ by~\eqref{eq:Nm}. Thus every edge is covered by the closed neighbourhood of a simplicial vertex, and $\Gamma$ is simplicially covered.

\emph{(ii)$\Rightarrow$(iii).} Assume $\Gamma$ is simplicially covered. By Lemma~\ref{lem:s-equiv} we may choose a set $I$ of simplicial vertices containing exactly one vertex of each s-equivalence class; $I$ is independent, and every simplicial vertex $w$ has a representative $\hat w\in I$ with $N[\hat w]=N[w]$. Put $S=I$ and
\[
A_v=\{w\in I: v\in N[w]\}=N[v]\cap I\qquad(v\in V).
\]
We check the three required properties.

\emph{Non-emptiness.} If $v$ is isolated, then $N[v]=\{v\}$ is a clique, so $v$ is simplicial and its class is $\{v\}$; hence $v\in I$ and $v\in A_v$. If $v$ is not isolated, pick an edge $vu$; by hypothesis there is a simplicial $w$ with $u,v\in N[w]$, and then $v\in N[\hat w]$, i.e.\ $\hat w\in A_v$.

\emph{Singletons.} For $w\in I$ we have $A_w=N[w]\cap I=\{w\}$, since $I$ is independent. Hence each singleton $\{w\}$, $w\in S$, occurs in the family.

\emph{Intersection graph.} Let $u\ne v$. If $A_u\cap A_v\ni w$, then $u,v\in N[w]$ and $N[w]$ is a clique, so $uv\in E$. Conversely if $uv\in E$, then some simplicial $w$ has $u,v\in N[w]=N[\hat w]$, and $\hat w\in A_u\cap A_v$. Therefore $\Gamma$ is the intersection graph of $(A_v)_{v\in V}$.

\emph{(iii)$\Rightarrow$(iv).} Let $(A_v)_{v\in V}$ and $S$ be as in (iii). For each $s\in S$ choose one vertex $v_s$ with $A_{v_s}=\{s\}$; the $v_s$ are distinct. Define a relation on $V$ by
\[
v_s<v\qquad\text{whenever }v\ne v_s\text{ and } s\in A_v,
\]
and no other strict relations. This is a strict partial order of height at most $1$: a vertex of the form $v_s$ is never above anything, since $v_s>v_t$ would require $t\in A_{v_s}=\{s\}$, i.e.\ $v_t=v_s$, contradicting $v_s\ne v_t$; hence there are no chains of length $2$ and transitivity and antisymmetry hold vacuously. Let $X$ be $V$ with the corresponding Alexandrov topology (a $T_0$ space with $|X|=|V|$). Its minimal points are exactly the $v_s$: every $v_s$ is minimal, and every other vertex $v$ has $A_v\neq\emptyset$, hence lies above $v_s$ for any $s\in A_v$. For every $v$ we get $M(v)=\{v_s: s\in A_v\}$ (for $v=v_t$ this reads $M(v_t)=\{v_t\}$, consistent with $A_{v_t}=\{t\}$). Therefore, for $u\ne v$,
\[
u\sim_{\Ins(X)}v\iff M(u)\cap M(v)\neq\emptyset\iff A_u\cap A_v\ne\emptyset\iff uv\in E,
\]
using the description of $\Ins$ at the start of Section~\ref{sec:finite} and the injectivity of $s\mapsto v_s$. Hence $\Ins(X)=\Gamma$.

\emph{(iv)$\Rightarrow$(i)} is trivial. This completes the proof of the equivalence.

\emph{The ``moreover'' statement.} Let $X$ be a finite $T_0$ space and identify $\Gamma$ with $\Ins(X)$. In a $T_0$ space the basic points are exactly the open points (if $U_m$ is minimal then $U_m=\{m\}$, since $U_z\subsetneq U_m$ for every $z<m$, and conversely). By the proof of (i)$\Rightarrow$(ii), every open point is simplicial. Distinct open points are non-adjacent ($U_m\cap U_{m'}=\{m\}\cap\{m'\}=\emptyset$), so no two open points are s-equivalent. It remains to see that every s-equivalence class contains an open point. Let $w$ be simplicial and choose an open point $m\le w$. Then $m\in U_w$, so $w\in N[m]$ by~\eqref{eq:Nm}, i.e.\ $m\in N[w]$. Now $m$ and $w$ are both simplicial and equal or adjacent, so they are s-equivalent. Hence $\Ob$ is a system of representatives of the s-equivalence classes and $|\Ob|=\sigma(\Gamma)$. The equality with $\alpha(\Gamma)$ and $\theta(\Gamma)$ is Proposition~\ref{prop:alpha}. The statement about non-open points follows since $|X|=|V|$. The version for $\Cm$ follows by applying everything to $X^{op}$, since $\Cm(X)=\Ins(X^{op})$ and the open points of $X^{op}$ are the closed points of $X$.
\end{proof}

\begin{corollary}\label{cor:nonrealizable}
\begin{enumerate}[label=\textup{(\alph*)}]
\item A path $P_n$ is an inseparability graph of a finite space iff $n\le3$; a cycle $C_n$ is one iff $n=3$.
\item A tree is an inseparability graph of a finite space iff its diameter is at most $2$ \textup{(}i.e.\ it is a star\textup{)}.
\item If $\Gamma$ is the inseparability graph of a finite space then $\alpha(\Gamma)=\theta(\Gamma)$. In particular no graph with $\alpha<\theta$ \textup{(}such as $C_5$\textup{)} is an inseparability graph.
\item A finite graph $\Gamma$ is the Urysohn graph of a finite space iff $\Gamma=\Delta\sq$ for a simplicially covered graph $\Delta$ on the same vertex set.
\end{enumerate}
\end{corollary}
\begin{proof}
(a),(b) In a tree the simplicial vertices are the leaves (and the vertices of a $K_1$ or $K_2$), and $N[\ell]=\{\ell,p\}$ for a leaf $\ell$ with neighbour $p$. So a tree is simplicially covered iff every edge contains a leaf, i.e.\ iff it has no edge between two non-leaves, i.e.\ iff it has diameter at most $2$. This gives (b) and the statement for paths. A cycle $C_n$ with $n\ge4$ has no simplicial vertex, while $C_3=K_3$ is realized by a space with a minimum. (c) follows from Theorem~\ref{thm:realization}. (d) follows from Theorem~\ref{thm:realization} and $\Ury(X)=\Ins(X)\sq$.
\end{proof}

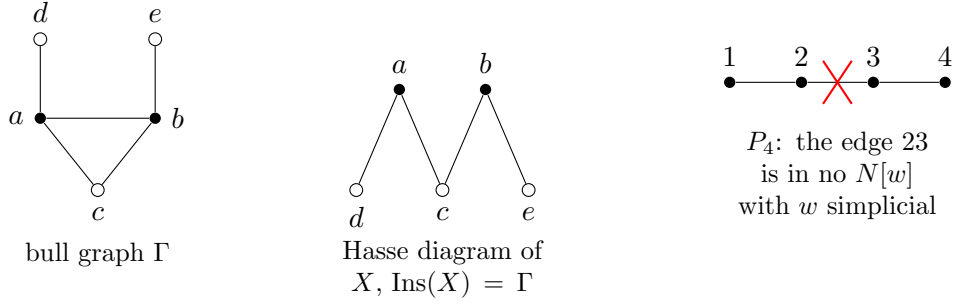
\begin{figure}[ht]
\centering
\begin{tikzpicture}[scale=0.95,pt/.style={circle,fill,inner sep=1.5pt},sp/.style={circle,draw,fill=white,inner sep=1.7pt}]
\begin{scope}
\node[pt,label=left:$a$] (a) at (0,1) {}; \node[pt,label=right:$b$] (b) at (1.6,1) {};
\node[sp,label=below:$c$] (c) at (0.8,0) {};
\node[sp,label=above:$d$] (d) at (0,2.1) {}; \node[sp,label=above:$e$] (e) at (1.6,2.1) {};
\draw (a)--(b)--(c)--(a) (d)--(a) (e)--(b);
\node[font=\small] at (0.8,-0.9) {bull graph $\Gamma$};
\end{scope}
\begin{scope}[xshift=4.4cm]
\node[sp,label=below:$d$] (d) at (0,0) {}; \node[sp,label=below:$c$] (c) at (1.2,0) {}; \node[sp,label=below:$e$] (e) at (2.4,0) {};
\node[pt,label=above:$a$] (a) at (0.6,1.4) {}; \node[pt,label=above:$b$] (b) at (1.8,1.4) {};
\draw (d)--(a)--(c)--(b)--(e);
\node[font=\small,text width=3.4cm,align=center] at (1.2,-1.1) {Hasse diagram of $X$, $\Ins(X)=\Gamma$};
\end{scope}
\begin{scope}[xshift=9.6cm]
\node[pt,label=above:$1$] (1) at (0,1.5) {}; \node[pt,label=above:$2$] (2) at (1,1.5) {};
\node[pt,label=above:$3$] (3) at (2,1.5) {}; \node[pt,label=above:$4$] (4) at (3,1.5) {};
\draw (1)--(2)--(3)--(4);
\draw[red,thick] (1.3,1.2)--(1.7,1.8) (1.3,1.8)--(1.7,1.2);
\node[text width=3.6cm,align=center,font=\small] at (1.5,0.2) {$P_4$: the edge $23$ is in no $N[w]$ with $w$ simplicial};
\end{scope}
\end{tikzpicture}
\caption{Theorem~\ref{thm:realization} in action. In the bull graph the simplicial vertices (hollow) are $c,d,e$ and their closed neighbourhoods cover all edges, so the bull is realized by the $5$-point space in the middle, whose open points are exactly $c,d,e$. Deleting $c$ yields $P_4$, which is not the inseparability graph of any finite space although it is an induced subgraph of one.}
\label{fig:bull}
\end{figure}

\begin{example}[Inseparability does not determine homotopy type]\label{ex:same-Ins}
Let $\Gamma=K_4-e$ with vertices $a,b,c,d$ and non-edge $bd$. The simplicial vertices are $b$ and $d$, and $N[b]=\{a,b,c\}$, $N[d]=\{a,c,d\}$ cover all edges. By Theorem~\ref{thm:realization} every $T_0$ realization has open points $b,d$, and $a,c$ lie above both. There are exactly two such posets: the pseudo-circle $S$ of Figure~\ref{fig:pseudocircle} ($a,c$ incomparable), which is weakly equivalent to $S^1$, and the space $T$ in which moreover $a<c$, which has a maximum and is contractible. Both have $\Ins\cong K_4-e$. In contrast with Corollary~\ref{cor:Sp-weak}, the inseparability graph does not determine the weak homotopy type. (Note $\Sp(S)=C_4$ while $\Sp(T)$ is $K_4-e$ as well.)
\end{example}

\begin{proposition}\label{prop:holes}
Let $X$ be a finite space and $k\ge 4$. If $\Ins(X)$ contains an induced cycle of length $k$, then $|X|\ge 2k$, and this bound is attained. Consequently the smallest finite spaces with non-chordal inseparability graph have $8$ points, and the smallest finite spaces with non-perfect inseparability graph have $10$ points.
\end{proposition}
\begin{proof}
By Theorem~\ref{thm:hierarchy}(e), $\Ins(X)$ is obtained from $\Ins(KX)$ by blowing up vertices into cliques of closed twins. Two vertices of an induced cycle of length $k\ge4$ cannot be closed twins (a common neighbour on the cycle would create a triangle), so the cycle maps to an induced cycle of length $k$ in $\Ins(KX)$, and $|KX|\le|X|$. Hence we may assume that $X$ is $T_0$. Let $x_1,\dots,x_k$ (indices mod $k$) induce a cycle. No $x_i$ is open: otherwise $M(x_i)=\{x_i\}$, and $x_{i-1},x_{i+1}$, both adjacent to $x_i$, would contain $x_i$ in $M(\cdot)$ and hence be adjacent, a chord since $k\ge4$. For each $i$ choose $m_i\in M(x_i)\cap M(x_{i+1})$. If $m_i=m_j$ with $i\ne j$, then $\{x_i,x_{i+1},x_j,x_{j+1}\}$ has at least three distinct elements, all containing $m_i$ in their $M$-sets and hence pairwise adjacent; but an induced cycle of length $\ge4$ contains no triangle. So $m_1,\dots,m_k$ are $k$ distinct open points, and together with the $k$ non-open $x_i$ we get $|X|\ge2k$. The height-$1$ space with open points $m_1,\dots,m_k$ and closed points $x_1,\dots,x_k$, where $x_i$ lies above $m_{i-1}$ and $m_i$, attains the bound, as its inseparability graph is $C_k$ on the $x_i$ plus the pendant vertices $m_i$ adjacent to $x_i,x_{i+1}$.

For perfectness, by the Strong Perfect Graph Theorem a non-perfect graph contains an odd hole $C_{2l+1}$ ($l\ge2$) or an odd antihole $\overline{C_{2l+1}}$ ($l\ge3$; note $\overline{C_5}\cong C_5$). An odd hole forces $|X|\ge 10$ by the first part, with equality for $C_5$. Suppose $\Ins(X)$ contains an induced $\overline{C_{n}}$, $n=2l+1\ge7$. No vertex of it is open, since the neighbourhood of an open point is a clique while in $\overline{C_n}$ each vertex has two non-adjacent neighbours. Each open point $m$ covers, via the clique ${\uparrow}m$, the edges of a clique of $\overline{C_n}$, i.e.\ of an independent set of $C_n$, which has at most $l$ vertices, and every edge of the antihole must be covered in this way (a common lower bound can be replaced by an open point below it). The antihole has $n(n-3)/2$ edges and a clique of size at most $l$ covers at most $\binom l2$ of them, so there are at least $\frac{n(n-3)/2}{\binom l2}=\frac{2(2l+1)}{l}>4$ open points, i.e.\ at least $5$. These open points are distinct from the $n$ vertices of the antihole, so $|X|\ge n+5\ge12$. Hence $10$ is the minimum.
\end{proof}

\section{Equality in the hierarchy for finite spaces}\label{sec:equality}

In the theory of graphs on groups, equality of two graphs in the hierarchy characterizes important classes of groups~\cite{Cam-groups,ACNS}. For finite spaces the corresponding classes are described by irreducibility of components and by chains. By Theorem~\ref{thm:hierarchy}(e) the graphs of $X$ are twin blow-ups of those of $KX$, and a blow-up of two graphs in the same way preserves equality or inequality; so it suffices to treat $T_0$ spaces. We say a finite space is a \emph{chain} if its specialization preorder is total.

\begin{theorem}\label{thm:equality}
Let $X$ be a finite $T_0$ space. Then:
\begin{enumerate}[label=\textup{(\arabic*)}]
\item $\Sp(X)=\Ins(X)$ iff $\cl\{x\}$ is a chain for every $x\in X$.
\item $\Sp(X)=\Cm(X)$ iff $U_x$ is a chain for every $x\in X$.
\item The following are equivalent: \textup{(a)} $\Ins(X)=\Ury(X)$; \textup{(b)} $\Ins(X)=\Fun(X)$; \textup{(c)} $|M(x)|=1$ for every $x$; \textup{(d)} every connected component of $X$ contains exactly one open point; \textup{(e)} every connected component of $X$ is irreducible.
\item The following are equivalent: \textup{(a)} $\Cm(X)=\Ury(X)$; \textup{(b)} $\Cm(X)=\Fun(X)$; \textup{(c)} $|M^*(x)|=1$ for every $x$; \textup{(d)} every connected component of $X$ contains exactly one closed point.
\item $\Ins(X)=\Cm(X)$ iff every connected component of $X$ contains exactly one open point and exactly one closed point; in this case $\Ins(X)=\Cm(X)=\Ury(X)=\Fun(X)$.
\item $\Ury(X)=\Fun(X)$ iff any two open points in the same component have intersecting closures, iff every connected component of $\Ins(X)$ has diameter at most $2$.
\item The following are equivalent: \textup{(a)} $\Sp(X)=\Fun(X)$; \textup{(b)} $\Sp(X)=\Ury(X)$; \textup{(c)} $\Sp(X)=\Ins(X)$ and $\Sp(X)=\Cm(X)$; \textup{(d)} every connected component of $X$ is a chain.
\end{enumerate}
\end{theorem}

\begin{proof}
We use the order-theoretic descriptions at the beginning of Section~\ref{sec:finite}, Lemma~\ref{lem:Ury-open}, and the fact that connected components of $X$ are the classes of the equivalence relation $\approx$ generated by $\le$ (Theorem~\ref{thm:hierarchy}(f)). Recall that $\cl\{x\}={\uparrow}x$ and $U_x={\downarrow}x$.

\emph{(1)} Suppose every ${\uparrow}z$ is a chain. If $x\sim_{\Ins}y$ then $x,y$ have a common lower bound $z$, so $x,y\in{\uparrow}z$ are comparable; thus $E(\Ins)\subseteq E(\Sp)$, and the reverse inclusion holds always. Conversely, suppose $\Sp(X)=\Ins(X)$ and let $y\ne y'$ lie in ${\uparrow}x$. Then $x$ is a common lower bound, so $y\sim_{\Ins}y'$, hence $y\sim_{\Sp}y'$: $y$ and $y'$ are comparable. So ${\uparrow}x$ is a chain.

\emph{(2)} Apply (1) to $X^{op}$, using $\Sp(X^{op})=\Sp(X)$ and $\Cm(X)=\Ins(X^{op})$.

\emph{(3)} We prove (b)$\Rightarrow$(a)$\Rightarrow$(c)$\Rightarrow$(d)$\Rightarrow$(b) and (d)$\Leftrightarrow$(e).

(b)$\Rightarrow$(a): $E(\Ins)\subseteq E(\Ury)\subseteq E(\Fun)$, so equality of the outer two forces equality with the middle one.

(a)$\Rightarrow$(c): Suppose $m\neq m'$ both lie in $M(x)$. Then $x$ is not open (for an open point $M(x)=\{x\}$), so $m,m'<x$. Since $m\in U_m\cap U_x$ and $m'\in U_x\cap U_{m'}$, we have $m\sim_{\Ins}x\sim_{\Ins}m'$, so $d_{\Ins}(m,m')\le2$ and $m\sim_{\Ury}m'$. But $U_m\cap U_{m'}=\{m\}\cap\{m'\}=\emptyset$, so $m\not\sim_{\Ins}m'$, contradicting (a). Hence $|M(x)|\le1$, and $|M(x)|\ge1$ always.

(c)$\Rightarrow$(d): Write $M(x)=\{\mu(x)\}$. If $x\le y$ then $\mu(x)\le x\le y$, so $\mu(x)\in M(y)$ and $\mu(y)=\mu(x)$. Hence $\mu$ is constant on the classes of the equivalence relation generated by $\le$, that is, on connected components. Let $C$ be a component and $m_C$ the common value of $\mu$ on $C$. Every open point $m\in C$ satisfies $m=\mu(m)=m_C$, so $C$ contains exactly one open point.

(d)$\Rightarrow$(b): Let $C$ be a component with unique open point $m$. For $x\in C$, $U_x\subseteq C$ (components are open) and $U_x$ contains an open point, which must be $m$. Hence $m\le x$ for all $x\in C$, so any two points of $C$ have the common lower bound $m$ and are $\Ins$-adjacent. Points in different components are not $\Fun$-adjacent. Therefore $\Ins(X)=\Fun(X)$.

(d)$\Leftrightarrow$(e): A component $C$ is itself a finite $T_0$ space (it is open and closed). By the proof of (d)$\Rightarrow$(b), if $C$ has a unique open point $m$ then $m$ is the minimum of $C$, i.e.\ $\cl\{m\}={\uparrow}m=C$ and $m$ is a generic point, so $C$ is irreducible. Conversely if $C=\cl\{g\}$ then $g\le x$ for all $x\in C$, so $g$ is the minimum and the only minimal element of $C$, i.e.\ the only open point in $C$.

\emph{(4)} Note that $X$ and $X^{op}$ have the same connected components, so $\Fun(X^{op})=\Fun(X)$, while $\Cm(X)=\Ins(X^{op})$, the closed points of $X$ are the open points of $X^{op}$, and $M^*(x)$ computed in $X$ is $M(x)$ computed in $X^{op}$. (The Urysohn graph, however, is not invariant under passing to $X^{op}$, so we argue directly for (a).) (b)$\Rightarrow$(a) follows from $E(\Cm)\subseteq E(\Ury)\subseteq E(\Fun)$. (a)$\Rightarrow$(c): if $M\ne M'$ lie in $M^*(x)$ then $x\le M,M'$, so $x$ is a common lower bound and $M\sim_{\Ins}M'$, hence $M\sim_{\Ury}M'$; but $\cl\{M\}\cap\cl\{M'\}=\{M\}\cap\{M'\}=\emptyset$, so $M\not\sim_{\Cm}M'$, contradicting (a). (c)$\Rightarrow$(d) and (d)$\Rightarrow$(b) are the implications (3)(c)$\Rightarrow$(3)(d) and (3)(d)$\Rightarrow$(3)(b) applied to $X^{op}$, whose proofs used only $\Ins$, $\Fun$, components and the sets $M(x)$.

\emph{(5)} If every component $C$ has a unique open point and a unique closed point, then by (3) and (4) both $\Ins(X)$ and $\Cm(X)$ equal $\Fun(X)$, hence each other, and $\Ury(X)$ is sandwiched between them. Conversely, suppose $\Ins(X)=\Cm(X)$. If $m\ne m'$ lie in $M(x)$ then $x$ is a common upper bound of $m,m'$, so $m\sim_{\Cm}m'$ and therefore $m\sim_{\Ins}m'$, which is impossible for distinct open points. Hence $|M(x)|=1$ for all $x$, and by (3) each component has exactly one open point. Dually, if $M\neq M'$ lie in $M^*(x)$ then $M\sim_{\Ins}M'$ (common lower bound $x$), so $M\sim_{\Cm}M'$, impossible for distinct closed points; by (4) each component has exactly one closed point.

\emph{(6)} Suppose $\Ury(X)=\Fun(X)$ and let $m\ne m'$ be open points in the same component. Then $m\sim_{\Fun}m'$, so $m\sim_{\Ury}m'$, and by Lemma~\ref{lem:Ury-open} (with $M(m)=\{m\}$, $M(m')=\{m'\}$) the closures of $m$ and $m'$ meet. Conversely, assume the condition and let $x\ne y$ lie in a component $C$. Choose $m\in M(x)$ and $m'\in M(y)$; both lie in $C$ since $C$ is open. Either $m=m'$ or by hypothesis $\cl\{m\}\cap\cl\{m'\}\neq\emptyset$; in both cases Lemma~\ref{lem:Ury-open} gives $x\sim_{\Ury}y$. So $\Ury(X)$ is complete on each component and equals $\Fun(X)$. The last equivalence holds because $\Ury(X)=\Ins(X)\sq$ and the components of $\Ins(X)$ are those of $X$ (Proposition~\ref{prop:distances}): $\Ins(X)\sq$ is complete on each component iff each component of $\Ins(X)$ has diameter at most $2$.

\emph{(7)} (d)$\Rightarrow$(a),(b),(c): if each component is a chain, $\Sp(X)$ is complete on each component, so $\Sp(X)=\Fun(X)$, and all intermediate graphs coincide with it. (a)$\Rightarrow$(d) is clear: $\Sp$ complete on a component means that the component is totally ordered. (b)$\Rightarrow$(d): if $\Sp(X)=\Ury(X)$ then, as $E(\Sp)\subseteq E(\Ins)\subseteq E(\Ury)$, also $\Ins(X)=\Ury(X)$; by (3) $\Ins(X)=\Fun(X)$, whence $\Sp(X)=\Fun(X)$ and (d) follows from (a)$\Rightarrow$(d).

(c)$\Rightarrow$(d): By (1) and (2) the sets ${\uparrow}x$ and ${\downarrow}x$ are chains for every $x$. Suppose some component $C$ contains two incomparable points, and choose incomparable $x,y\in C$ with $d_{\Sp}(x,y)=k$ minimal; $k\ge2$ since they are incomparable, and $k<\infty$ by Proposition~\ref{prop:distances}. Let $x=z_0,z_1,z_2,\dots,z_k=y$ be a geodesic in $\Sp(X)$. Since it is a geodesic, $z_0$ and $z_2$ are not adjacent in $\Sp(X)$, i.e.\ they are incomparable. Consider the four possible configurations of $z_0,z_1,z_2$. If $z_0<z_1>z_2$, then $z_0,z_2\in{\downarrow}z_1$, a chain, so they are comparable. If $z_0>z_1<z_2$, then $z_0,z_2\in{\uparrow}z_1$, a chain, contradiction. If $z_0<z_1<z_2$ or $z_0>z_1>z_2$, transitivity makes $z_0,z_2$ comparable. All cases are contradictory, so every component is a chain.
\end{proof}

\begin{remark}
Part (3) is the topological analogue of the characterization of EPPO groups by the equality of the power graph and the enhanced power graph~\cite[Proposition~3.2]{Cam-groups}. A notable feature, with no counterpart for groups, is the ``collapse'' (3)(a)$\Rightarrow$(3)(b): equality of $\Ins$ with the next graph up forces equality with the top of the hierarchy. The conditions are reminiscent of algebraic geometry: the prime spectrum of an integral domain has a generic point (it is irreducible), that of a local ring has a unique closed point, and condition (5) asks for both on each component.
\end{remark}

\subsection{Cographs and chordality}\label{subsec:cograph}

Recall that a graph is a \emph{cograph} if it has no induced path $P_4$, and \emph{chordal} if it has no induced cycle of length at least $4$; a graph which is both is called \emph{trivially perfect}. For finite groups, Bubboloni, Fumagalli and Praeger~\cite{BFP} recently proved that if the enhanced power graph of a finite group is a cograph, then it is also chordal; their analysis is based on simplicial vertices. Since the closure-meeting graph $\Cm(X)$ is our analogue of the enhanced power graph (Table~\ref{tab:dictionary}), it is natural to ask whether the same phenomenon occurs for spaces. Proposition~\ref{prop:cograph}(b),(c) below shows that it does, for $\Cm(X)$ for every space and for $\Ins(X)$ for every Alexandrov space. We use two elementary facts.

\begin{enumerate}[label=(F\arabic*),leftmargin=3em]
\item\label{F1} For $k\ge5$ the cycle $C_k$ contains an induced $P_4$. Hence a cograph is chordal if and only if it has no induced $C_4$.
\item\label{F2} An induced $P_4$ or an induced cycle of length at least $4$ contains no two closed twins. Consequently, if $\Delta$ is obtained from a graph $\Delta_0$ by replacing vertices by cliques of closed twins, then $\Delta$ is a cograph \textup{(}respectively chordal\textup{)} if and only if $\Delta_0$ is. By Theorem~\ref{thm:hierarchy}\textup{(e)}, for each of the five graph types $\Gamma$ the graph $\Gamma(X)$ is a cograph \textup{(}respectively chordal\textup{)} if and only if $\Gamma(KX)$ is.
\end{enumerate}

For \ref{F2}: if $x,y$ are closed twins lying on an induced $P_4$ or on an induced cycle of length $\ge4$, then $x\sim y$. Since the subgraph is connected and has at least four vertices, some further vertex $z$ of it is adjacent to $x$ or to $y$, hence to both, and $x,y,z$ form a triangle, which is impossible. Hence such a subgraph of $\Delta$ maps injectively onto an induced copy in $\Delta_0$, while $\Delta_0$ is an induced subgraph of $\Delta$.

\begin{proposition}\label{prop:cograph}
Let $X$ be a topological space.
\begin{enumerate}[label=\textup{(\alph*)}]
\item If $\Sp(X)=\Ins(X)$ or $\Sp(X)=\Cm(X)$, then $\Sp(X)$ is a cograph and is chordal \textup{(}that is, trivially perfect\textup{)}.
\item If $\Cm(X)$ is a cograph, then $\Cm(X)$ is chordal.
\item If $X$ is an Alexandrov space \textup{(}in particular, if $X$ is finite\textup{)} and $\Ins(X)$ is a cograph, then $\Ins(X)$ is chordal.
\end{enumerate}
\end{proposition}

\begin{proof}
By \ref{F2} we may replace $X$ by $KX$: equalities between graphs are preserved and reflected by the twin blow-up, and so are the properties of being a cograph or chordal. So we assume that $X$ is $T_0$, i.e.\ that the specialization preorder $\le$ is a partial order. Recall that $\cl\{x\}={\uparrow}x=\{y:x\le y\}$ for every space.

\emph{(a)} Suppose first that $\Sp(X)=\Ins(X)$. We claim that ${\uparrow}x$ is a chain for every $x$. Let $y,z\in{\uparrow}x$. Since $x\le y$, every open set containing $y$ contains $x$, and likewise for $z$. Hence every neighbourhood of $y$ meets every neighbourhood of $z$ (in the point $x$), so $y\eqadj{\Ins}z$, hence $y\eqadj{\Sp}z$, i.e.\ $y$ and $z$ are comparable.

Suppose that $a-b-c-d$ is an induced path in $\Sp(X)$. Reversing the labelling if necessary we may assume $b<c$. If $a<b$, then $a<c$ by transitivity and $a\sim c$, which is false; so $b<a$. Then $a,c\in{\uparrow}b$, which is a chain, so $a\sim c$, a contradiction. Hence $\Sp(X)$ is a cograph.

Suppose that $x_1,\dots,x_k$ ($k\ge4$) induce a cycle in $\Sp(X)$, and let $b$ be a minimal element of the finite poset $\{x_1,\dots,x_k\}$. The two neighbours $a,c$ of $b$ on the cycle are comparable with $b$, and by minimality they are not below $b$; so $a,c\in{\uparrow}b$. Since ${\uparrow}b$ is a chain, $a\sim c$, which is a chord. Hence $\Sp(X)$ is chordal.

If $\Sp(X)=\Cm(X)$, the same argument shows that ${\downarrow}x$ is a chain for every $x$: if $y,z\le x$ then $x\in\cl\{y\}\cap\cl\{z\}$, so $y\eqadj{\Cm}z$ and hence $y,z$ are comparable. Since the two arguments above are symmetric under reversing the order (use a maximal element of the cycle instead of a minimal one), $\Sp(X)$ is again a cograph and chordal.

\emph{(b)} Since $\cl\{x\}={\uparrow}x$, two points are adjacent in $\Cm(X)$ if and only if they have a common upper bound. By \ref{F1} it suffices to show that an induced $C_4$ in $\Cm(X)$ yields an induced $P_4$. Let $x_0,x_1,x_2,x_3$ induce a $4$-cycle, with indices modulo $4$, so that $x_i\sim x_{i+1}$ while $x_0\not\sim x_2$ and $x_1\not\sim x_3$. For each $i$ choose a common upper bound $u_i$ of $x_i$ and $x_{i+1}$. We record three observations.

\smallskip
\noindent(1) \emph{$u_i\notin\{x_{i+2},x_{i+3}\}$, and $u_i$ is adjacent to neither $x_{i+2}$ nor $x_{i+3}$.} If $u_i=x_{i+2}$ then $x_i\le x_{i+2}$, so $x_{i+2}$ is a common upper bound of $x_i,x_{i+2}$ and $x_i\sim x_{i+2}$, which is false. If $u_i\sim x_{i+2}$, a common upper bound $w$ of $u_i$ and $x_{i+2}$ satisfies $w\ge u_i\ge x_i$, so again $x_i\sim x_{i+2}$. The statements for $x_{i+3}$ follow in the same way, using $x_{i+1}$ and $x_{i+3}$.

\smallskip
\noindent(2) \emph{$u_i\neq u_{i+2}$ and $u_i\not\sim u_{i+2}$.} Otherwise there is a common upper bound of $u_i$ and $u_{i+2}$, which lies above $x_i$ and $x_{i+2}$, giving $x_i\sim x_{i+2}$.

\smallskip
\noindent(3) \emph{Two candidate paths.} Consider
\[
\pi_i:\ u_i - x_{i+1} - x_{i+2} - u_{i+2}\qquad\text{and}\qquad \rho_i:\ u_i - x_i - x_{i+3} - u_{i+2}.
\]
By (1) and (2), the four vertices of $\pi_i$ are distinct and $\pi_i$ is an induced $P_4$ as soon as $u_i\neq x_{i+1}$ and $u_{i+2}\neq x_{i+2}$. Indeed, the three edges are present because $u_i\ge x_{i+1}$ and $u_{i+2}\ge x_{i+2}$ (a point is adjacent to any distinct point above it), and the three non-edges $u_i x_{i+2}$, $x_{i+1}u_{i+2}$, $u_iu_{i+2}$ are given by (1) and (2). Similarly, $\rho_i$ is an induced $P_4$ as soon as $u_i\ne x_i$ and $u_{i+2}\neq x_{i+3}$.

\smallskip
Suppose, for a contradiction, that $\pi_0,\rho_0,\pi_1,\rho_1$ all fail to be induced paths. Since $u_0$ cannot equal both $x_0$ and $x_1$, and $u_2$ cannot equal both $x_2$ and $x_3$, the failure of $\pi_0$ and $\rho_0$ means that either $u_0=x_1$ and $u_2=x_3$, or $u_0=x_0$ and $u_2=x_2$. In the first case $x_0<x_1$ and $x_2<x_3$; in the second case $x_1<x_0$ and $x_3<x_2$. In the same way, the failure of $\pi_1$ and $\rho_1$ means that either $x_1<x_2$ and $x_3<x_0$, or $x_2<x_1$ and $x_0<x_3$. We examine the four combinations.
\begin{itemize}[leftmargin=2em]
\item $x_0<x_1<x_2$: then $x_0<x_2$, so $x_0\sim x_2$.
\item $x_0<x_3$ and $x_2<x_3$: then $x_3$ is a common upper bound of $x_0,x_2$, so $x_0\sim x_2$.
\item $x_1<x_0$ and $x_3<x_0$: then $x_0$ is a common upper bound of $x_1,x_3$, so $x_1\sim x_3$.
\item $x_3<x_2<x_1<x_0<x_3$: this contradicts antisymmetry.
\end{itemize}
Each case is impossible. Hence one of $\pi_0,\rho_0,\pi_1,\rho_1$ is an induced $P_4$. Therefore, if $\Cm(X)$ is a cograph, it contains no induced $C_4$, and by \ref{F1} it is chordal.

\emph{(c)} If $X$ is Alexandrov, then by Theorem~\ref{thm:hierarchy}(f) two points are adjacent in $\Ins(X)$ if and only if they have a common lower bound. Thus $\Ins(X)$ is the common-upper-bound graph of the opposite order, i.e.\ $\Ins(X)=\Cm(X^{op})$, where $X^{op}$ is the Alexandrov space of the opposite preorder. Now apply (b) to $X^{op}$.
\end{proof}

\begin{remark}\label{rem:cograph-counter}
\begin{enumerate}[label=(\roman*),leftmargin=2em]
\item \emph{The Alexandrov hypothesis in Proposition~\ref{prop:cograph}(c) cannot be omitted.} Apply Theorem~\ref{thm:XGamma} with $\Gamma=C_4$. The space $X_{C_4}$ is a second countable $T_1$ space, and in fact a non-Hausdorff $1$-manifold by Theorem~\ref{thm:XGamma}(e), since $C_4$ is $2$-regular. Its inseparability graph is $C_4$ together with isolated vertices. This graph is a cograph, because $C_4\cong K_{2,2}$ and isolated vertices create no induced $P_4$, but it is not chordal.
\item \emph{The analogue of Proposition~\ref{prop:cograph}(b),(c) fails for $\Sp$.} For the pseudo-circle $S$ of Figure~\ref{fig:pseudocircle} we have $\Sp(S)\cong C_4$, a cograph which is not chordal.
\item \emph{The converse of Proposition~\ref{prop:cograph}(a) is false.} The pseudo-circle again shows this: $\Sp(S)$ is a cograph, but $\Sp(S)$ differs from both $\Ins(S)$ and $\Cm(S)$. Even a chordal cograph $\Sp(X)$ need not equal $\Ins(X)$ or $\Cm(X)$. Take the disjoint union $X=V\sqcup V^{op}$ of the three-point spaces of Remark~\ref{rem:Sp-not-determine}. Then $\Sp(X)\cong P_3\sqcup P_3$ is trivially perfect. The two open points of $V^{op}$ have a common upper bound, so they are adjacent in $\Cm(X)$; they are incomparable, so they are not adjacent in $\Sp(X)$. Hence $\Sp(X)\neq\Cm(X)$. Dually, the two closed points of $V$ show that $\Sp(X)\ne\Ins(X)$.
\item \emph{Finite spaces.} A finite space can have an inseparability graph which is not a cograph; the smallest examples have $5$ points (Section~\ref{sec:data}). By Proposition~\ref{prop:holes}, a finite space whose inseparability graph is not chordal has at least $8$ points, and Proposition~\ref{prop:cograph}(c) shows that such a graph is never a cograph. The same holds for $\Cm(X)$ by passing to $X^{op}$.
\item \emph{Comparison with groups.} In~\cite{BFP} the finite groups whose enhanced power graph is a cograph are characterized, and the nonabelian finite simple groups with this property are classified. In the topological setting, Proposition~\ref{prop:cograph} shows that for $\Cm$, and for $\Ins$ of Alexandrov spaces, the cograph property already forces trivial perfectness.
\end{enumerate}
\end{remark}

\section{Homotopy: Dowker duality for inseparability}\label{sec:homotopy}

The clique complex of $\Sp(X)$ is the order complex of $X$ (Proposition~\ref{prop:Sp}). The inseparability and closure-meeting graphs carry natural higher-dimensional structures, namely the nerves of the families of minimal open sets and of point closures.

\begin{definition}
Let $X$ be a finite space. The \emph{inseparability complex} $\Nins(X)$ is the simplicial complex whose simplices are the non-empty $\sigma\subseteq X$ with $\bigcap_{x\in\sigma}U_x\ne\emptyset$. The \emph{closure-meeting complex} $\Ncm(X)$ has simplices the non-empty $\sigma$ with $\bigcap_{x\in\sigma}\cl\{x\}\neq\emptyset$. For a finite $T_0$ space, the \emph{Dowker complex} $\Dow(X)$ is the simplicial complex on $\Ob$ whose simplices are the non-empty sets $S\subseteq\Ob$ with $\bigcap_{m\in S}\cl\{m\}\ne\emptyset$, and $\Dow^*(X)$ is the complex on $\Cb$ whose simplices are the non-empty $T\subseteq\Cb$ with $\bigcap_{c\in T}U_c\ne\emptyset$.
\end{definition}

The $1$-skeleton of $\Nins(X)$ is $\Ins(X)$ and that of $\Ncm(X)$ is $\Cm(X)$. Moreover $\Nins(X)\subseteq\Cl(\Ins(X))$, and in general the inclusion is proper (Example~\ref{ex:triangle}).

\begin{theorem}\label{thm:homotopy}
Let $X$ be a finite $T_0$ space.
\begin{enumerate}[label=\textup{(\alph*)}]
\item There are homotopy equivalences
\[
|\Nins(X)|\simeq|\Dow(X)|\simeq|\Dow^*(X)|\simeq|\Ncm(X)| .
\]
In particular $|\Nins(X)|\simeq|\Nins(X^{op})|$, although $\Ins(X)$ and $\Ins(X^{op})=\Cm(X)$ may be non-isomorphic.
\item Suppose that for every non-empty set $S\subseteq\Ob$ with $F_S:=\bigcap_{m\in S}\cl\{m\}\neq\emptyset$, the order complex $\Delta(F_S)$ is contractible \textup{(}this holds, for instance, if every such $F_S$ has a minimum, i.e.\ is irreducible\textup{)}. Then $|\Delta(X)|\simeq|\Dow(X)|$, and consequently $|\Nins(X)|$ and $|\Ncm(X)|$ are weakly homotopy equivalent to $X$.
\item $\Nins(X)=\Cl(\Ins(X))$ if and only if the family $\{M(x):x\in X\}$ has the Helly property \textup{(}every pairwise intersecting subfamily has non-empty intersection\textup{)}.
\end{enumerate}
\end{theorem}

\begin{proof}
\emph{Preliminary observations.} Since $U_x={\downarrow}x$ and $\cl\{x\}={\uparrow}x$, a set $\sigma$ is a simplex of $\Nins(X)$ iff its elements have a common lower bound, iff they have a common \emph{open} lower bound (below any lower bound there is a minimal element). Dually $\sigma\in\Ncm(X)$ iff its elements have a common closed upper bound. Similarly $S\in\Dow(X)$ iff the open points in $S$ have a common upper bound iff they have a common closed upper bound; and $T\in\Dow^*(X)$ iff the closed points in $T$ have a common open lower bound.

\emph{(a)} We apply Dowker's Lemma~\ref{lem:dowker} three times.

First, let $R_1\subseteq X\times\Ob$ be the relation $xR_1m\iff m\le x$. Its complex on $X$ is
\[
K_{R_1}=\{\sigma: \exists m\in\Ob\ \forall x\in\sigma\ m\le x\}=\Nins(X)
\]
by the preliminary observation, and its complex on $\Ob$ is
\[
L_{R_1}=\{S\subseteq\Ob:\exists x\in X\ \forall m\in S\ m\le x\}=\Dow(X).
\]
Every element of $X$ and of $\Ob$ is related to something ($x$ lies above an open point; $m\le m$), so no discarding is needed. Lemma~\ref{lem:dowker} gives $|\Nins(X)|\simeq|\Dow(X)|$.

Second, let $R_3\subseteq\Ob\times\Cb$ be $mR_3c\iff m\le c$. Then $K_{R_3}=\{S\subseteq\Ob:\exists c\in\Cb\ \forall m\in S\ m\le c\}=\Dow(X)$, because a family of open points has a common upper bound iff it has a common closed upper bound; and $L_{R_3}=\{T\subseteq\Cb:\exists m\in\Ob\ \forall c\in T\ m\le c\}=\Dow^*(X)$. Hence $|\Dow(X)|\simeq|\Dow^*(X)|$.

Third, let $R_2\subseteq X\times\Cb$ be $xR_2c\iff x\le c$. Then $K_{R_2}=\Ncm(X)$ and $L_{R_2}=\{T\subseteq\Cb:\exists x\ \forall c\in T\ x\le c\}=\Dow^*(X)$, using that a common lower bound can be replaced by a common open lower bound. Hence $|\Ncm(X)|\simeq|\Dow^*(X)|$.

Chaining the three equivalences proves the first assertion. Since $\Nins(X^{op})=\Ncm(X)$ (open and closed sets are interchanged), we get $|\Nins(X)|\simeq|\Nins(X^{op})|$. The pseudo-circle $S$ has $\Ins(S)\cong\Cm(S)$, but for the space $T$ of Example~\ref{ex:same-Ins} we have $\Ins(T)\cong K_4-e$ whereas $\Cm(T)=K_4$ (as $T$ has a maximum); so $\Ins(X)\not\cong\Ins(X^{op})$ in general.

\emph{(b)} We show first that
\begin{equation}\label{eq:cover}
\Delta(X)=\bigcup_{m\in\Ob}\Delta(\cl\{m\}),
\end{equation}
where $\Delta(\cl\{m\})$ is the order complex of the subposet ${\uparrow}m$, a subcomplex of $\Delta(X)$. The inclusion $\supseteq$ is clear. Conversely let $x_1<x_2<\dots<x_r$ be a chain, and choose an open point $m\le x_1$; then all $x_i\ge m$, so the chain lies in ${\uparrow}m$. Next, for $\emptyset\ne S\subseteq\Ob$,
\begin{equation}\label{eq:capcomplex}
\bigcap_{m\in S}\Delta(\cl\{m\})=\Delta\Big(\bigcap_{m\in S}\cl\{m\}\Big)=\Delta(F_S),
\end{equation}
because a chain lies in each ${\uparrow}m$ iff all its elements lie in their intersection. Each $\Delta(\cl\{m\})$ is contractible, being a cone with apex $m$ (the minimum of ${\uparrow}m$). By hypothesis each non-empty intersection~\eqref{eq:capcomplex} is contractible. The nerve of the cover~\eqref{eq:cover} is the complex on $\Ob$ whose simplices are the sets $S$ with $\Delta(F_S)\ne\emptyset$, i.e.\ with $F_S\ne\emptyset$; this is exactly $\Dow(X)$. The nerve lemma (Lemma~\ref{lem:nerve}) gives $|\Delta(X)|\simeq|\Dow(X)|$. Combining with (a) and McCord's theorem, $|\Nins(X)|\simeq|\Dow(X)|\simeq|\Delta(X)|\simeq_wX$, and similarly for $\Ncm(X)$. If $F_S$ has a minimum $z$, then $\Delta(F_S)$ is a cone with apex $z$ and hence contractible; and a finite $T_0$ space has a minimum iff it is irreducible.

\emph{(c)} A set $\sigma$ is a simplex of $\Cl(\Ins(X))$ iff $M(x)\cap M(y)\ne\emptyset$ for all $x,y\in\sigma$, and a simplex of $\Nins(X)$ iff $\bigcap_{x\in\sigma}M(x)\ne\emptyset$ (by the preliminary observation). The two complexes coincide iff every pairwise intersecting subfamily of $\{M(x)\}$ has a common element, i.e.\ iff the family is Helly.
\end{proof}

\begin{corollary}\label{cor:faceposet}
Let $K$ be a finite simplicial complex and $\chi(K)$ its face poset, viewed as a finite $T_0$ space \textup{(}open sets are down-sets; the open points are the vertices\textup{)}. Then $\Dow(\chi(K))=K$, and
\[
|\Nins(\chi(K))|\simeq|\Ncm(\chi(K))|\simeq|K| .
\]
In words: two faces of $K$ are inseparable iff they share a vertex, and the complex of families of faces sharing a common vertex has the homotopy type of $K$.
\end{corollary}
\begin{proof}
The open points of $\chi(K)$ are the vertices $\{v\}$. For a set $S$ of vertices, $F_S=\bigcap_{v\in S}\cl\{v\}$ is the set of faces containing all $v\in S$. It is empty unless $S$ is a face, in which case it is ${\uparrow}S$, which has minimum $S$. So the hypothesis of Theorem~\ref{thm:homotopy}(b) holds, and $S\in\Dow(\chi(K))$ iff $F_S\neq\emptyset$ iff $S\in K$. Hence $\Dow(\chi(K))=K$, and Theorem~\ref{thm:homotopy} gives $|\Nins(\chi(K))|\simeq|\Ncm(\chi(K))|\simeq|K|$. (This is consistent with McCord's theorem, as $|\Delta(\chi(K))|$ is the barycentric subdivision of $|K|$.)
\end{proof}

\begin{example}[The boundary of a triangle]\label{ex:triangle}
Let $K=\partial\Delta^2$, with vertices $a,b,c$ and edges $ab,bc,ca$, so $\chi(K)$ is a $6$-point model of $S^1$. The inseparability graph $\Ins(\chi(K))$ has edges $a$--$ab$, $a$--$ca$, $b$--$ab$, $b$--$bc$, $c$--$bc$, $c$--$ca$ and the triangle $ab$--$bc$--$ca$ (any two edges of $K$ share a vertex). The maximal simplices of $\Nins(\chi(K))$ are $\{a,ab,ca\}$, $\{b,ab,bc\}$, $\{c,bc,ca\}$; the triangle $\{ab,bc,ca\}$ is \emph{not} a simplex because the three edges have no common vertex. Hence $|\Nins(\chi(K))|$ is an annulus, homotopy equivalent to $S^1$, in accordance with Corollary~\ref{cor:faceposet}, while $\Cl(\Ins(\chi(K)))$ contains the extra $2$-simplex $\{ab,bc,ca\}$ and is collapsible, hence contractible. The family $\{M(x)\}$ fails the Helly property: $M(ab)=\{a,b\}$, $M(bc)=\{b,c\}$, $M(ca)=\{c,a\}$. See Figure~\ref{fig:triangle}.
\end{example}

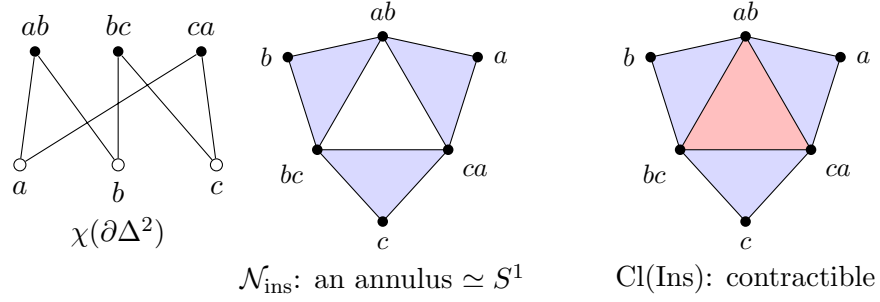
\begin{figure}[ht]
\centering
\begin{tikzpicture}[scale=1.0,pt/.style={circle,fill,inner sep=1.4pt},op/.style={circle,draw,fill=white,inner sep=1.6pt}]
\begin{scope}
\node[op,label=below:$a$] (a) at (0,0) {}; \node[op,label=below:$b$] (b) at (1.3,0) {}; \node[op,label=below:$c$] (c) at (2.6,0) {};
\node[pt,label=above:$ab$] (ab) at (0.2,1.5) {}; \node[pt,label=above:$bc$] (bc) at (1.3,1.5) {}; \node[pt,label=above:$ca$] (ca) at (2.4,1.5) {};
\draw (a)--(ab)--(b)--(bc)--(c)--(ca)--(a);
\node at (1.3,-0.9) {$\chi(\partial\Delta^2)$};
\end{scope}
\begin{scope}[xshift=4.8cm,yshift=0.7cm]
\coordinate (AB) at (90:1); \coordinate (BC) at (210:1); \coordinate (CA) at (330:1);
\coordinate (B) at ($(AB)!0.5!(BC)+(150:0.95)$); \coordinate (C) at ($(BC)!0.5!(CA)+(270:0.95)$); \coordinate (A) at ($(CA)!0.5!(AB)+(30:0.95)$);
\fill[blue!15] (AB)--(BC)--(B)--cycle; \fill[blue!15] (BC)--(CA)--(C)--cycle; \fill[blue!15] (CA)--(AB)--(A)--cycle;
\draw (AB)--(BC)--(CA)--cycle (AB)--(B)--(BC) (BC)--(C)--(CA) (CA)--(A)--(AB);
\foreach \p/\l/\pos in {AB/ab/above,BC/bc/below left,CA/ca/below right,A/a/right,B/b/left,C/c/below}{\node[pt,label={[font=\small]\pos:$\l$}] at (\p) {};}
\node at (0,-2.2) {$\Nins$: an annulus $\simeq S^1$};
\end{scope}
\begin{scope}[xshift=9.6cm,yshift=0.7cm]
\coordinate (AB) at (90:1); \coordinate (BC) at (210:1); \coordinate (CA) at (330:1);
\coordinate (B) at ($(AB)!0.5!(BC)+(150:0.95)$); \coordinate (C) at ($(BC)!0.5!(CA)+(270:0.95)$); \coordinate (A) at ($(CA)!0.5!(AB)+(30:0.95)$);
\fill[blue!15] (AB)--(BC)--(B)--cycle; \fill[blue!15] (BC)--(CA)--(C)--cycle; \fill[blue!15] (CA)--(AB)--(A)--cycle;
\fill[red!25] (AB)--(BC)--(CA)--cycle;
\draw (AB)--(BC)--(CA)--cycle (AB)--(B)--(BC) (BC)--(C)--(CA) (CA)--(A)--(AB);
\foreach \p/\l/\pos in {AB/ab/above,BC/bc/below left,CA/ca/below right,A/a/right,B/b/left,C/c/below}{\node[pt,label={[font=\small]\pos:$\l$}] at (\p) {};}
\node at (0,-2.2) {$\Cl(\Ins)$: contractible};
\end{scope}
\end{tikzpicture}
\caption{Example~\ref{ex:triangle}: the face poset of $\partial\Delta^2$ (open points hollow), its inseparability complex and the clique complex of its inseparability graph. The shaded middle triangle belongs only to the clique complex.}
\label{fig:triangle}
\end{figure}

\begin{example}[The hypothesis of Theorem~\ref{thm:homotopy}(b) is needed]
For the pseudo-circle $S$ (Figure~\ref{fig:pseudocircle}), $\Ob=\{a,b\}$ and $F_{\{a,b\}}=\{c,d\}$ is a two-point discrete space, which is not contractible. Here $\Dow(S)$ is the full simplex on $\{a,b\}$ (they have a common upper bound), so $|\Nins(S)|$ and $|\Ncm(S)|$ are contractible, whereas $S\simeq_w S^1$. Thus the inseparability complex remembers the homotopy type exactly when intersections of closures of open points are homotopically trivial; this is a finite-space version of the nerve theorem for good covers.
\end{example}

\begin{remark}
Simplicial complexes attached to algebraic structures, namely the independence and strong independence complexes of an algebra, are studied in~\cite{Cam-alg}; the inseparability and closure-meeting complexes play a similar role for spaces. Theorem~\ref{thm:homotopy}(a) is the topological counterpart of the ``dual pair'' phenomenon of~\cite[\S12]{Cam-groups}: there, a graph on the non-identity elements of a group and an intersection graph of subgroups are the two halved graphs of a bipartite incidence graph, and they share connectivity. Here the incidence relation between open points and closed points yields the two Dowker complexes $\Dow(X)$ and $\Dow^*(X)$, which share their entire homotopy type, and the inseparability and closure-meeting complexes are attached to them by two further Dowker dualities. Passing to $1$-skeleta, Proposition~\ref{prop:distances} recovers the equality of the numbers of components of $\Ins(X)$ and $\Cm(X)$.
\end{remark}

\section{Invariant topologies, primitive groups and neighbourhood filters}\label{sec:groups}

So far the group acting on a space has been its full homeomorphism group, which acts on each of the five graphs. In this section we reverse the point of view, following Cameron's work on topologies preserved by permutation groups~\cite{Cam-top,CamV}: we start with a permutation group $G$ on a set $\Omega$ and ask what the five graphs of a $G$-invariant topology can look like. Two observations make our hierarchy particularly well suited to this. First, the graph used by Cameron in his simplification of the Macpherson--Praeger theorem~\cite{MacPr,Cam-top}, in which $x$ and $y$ are joined if they have disjoint open neighbourhoods, is precisely the complement of the inseparability graph; we call it the \emph{separation graph} $\operatorname{Sep}(X)=\overline{\Ins(X)}$. Secondly, the correspondence between topologies and preorders used in~\cite{CamV} is the specialization preorder of Section~\ref{sec:prelim}, and the specialization graph is its comparability graph.

\subsection{Primitivity and separation}
Let $G$ be a permutation group on a set $\Omega$ with $|\Omega|\ge2$. A topology $\mathcal T$ on $\Omega$ is \emph{$G$-invariant} if $G\le\Homeo(\Omega,\mathcal T)$. Following~\cite{CamV}, a relation or topology on $\Omega$ is \emph{trivial} if it is invariant under the full symmetric group $\operatorname{Sym}(\Omega)$ (for instance the discrete, the indiscrete and the cofinite topologies are trivial). The group $G$ is \emph{primitive} if it is transitive and preserves no non-trivial equivalence relation, and \emph{strongly primitive} (Wielandt) if it is transitive and preserves no non-trivial preorder.

\begin{lemma}\label{lem:trivial-rel}
If $|\Omega|\ge2$, the only $\operatorname{Sym}(\Omega)$-invariant reflexive relations on $\Omega$ are equality and the universal relation. In particular these are the only trivial equivalence relations and the only trivial preorders.
\end{lemma}
\begin{proof}
$\operatorname{Sym}(\Omega)$ is transitive on ordered pairs of distinct points, so an invariant relation containing one such pair contains all of them.
\end{proof}

Recall (Section~\ref{sec:prelim}) that $\mathcal T$ is $T_0$ iff its indistinguishability relation is equality, and $T_1$ iff its specialization preorder is equality. For a preorder $P$ on $\Omega$ let $\mathcal A(P)$ denote its Alexandrov topology, whose open sets are the down-sets of $P$ (in our convention $x\le y$ means $x\in U_y$). The specialization preorder of $\mathcal A(P)$ is $P$ itself. In the terminology of~\cite{CamV}, the Alexandrov topologies are exactly the \emph{relational} topologies, those recovered from their specialization preorder. (Cameron writes $x\to y$ if every open set containing $x$ contains $y$; thus $x\to y$ iff $y\le x$ in our notation.)

The next proposition is due to Cameron~\cite[Proposition~2.1]{CamV}; we include a proof in our language because its mechanism is exactly the Kolmogorov and specialization structure of Theorem~\ref{thm:hierarchy}.

\begin{proposition}[\cite{CamV}]\label{prop:primitive-T0}
Let $G$ be a transitive permutation group on $\Omega$, $|\Omega|\ge2$.
\begin{enumerate}[label=\textup{(\alph*)}]
\item $G$ is primitive iff every non-trivial $G$-invariant topology on $\Omega$ is $T_0$.
\item $G$ is strongly primitive iff every non-trivial $G$-invariant topology on $\Omega$ is $T_1$, i.e.\ iff $\Sp(\Omega,\mathcal T)$ \textup{(}equivalently $\Cm(\Omega,\mathcal T)$\textup{)} is null for every non-trivial $G$-invariant topology $\mathcal T$.
\end{enumerate}
\end{proposition}

\begin{proof}
(a) Let $\mathcal T$ be $G$-invariant. Homeomorphisms preserve topological indistinguishability, so its indistinguishability relation $\equiv_{\mathcal T}$ is a $G$-invariant equivalence relation. If $G$ is primitive, $\equiv_{\mathcal T}$ is trivial, hence (Lemma~\ref{lem:trivial-rel}) equality, in which case $\mathcal T$ is $T_0$, or universal, in which case every non-empty open set is $\Omega$ and $\mathcal T$ is indiscrete, hence trivial. Conversely, suppose $G$ preserves a non-trivial equivalence relation $E$. The topology $\mathcal T_E$ whose open sets are the unions of $E$-classes is $G$-invariant and its indistinguishability relation is $E$, which is not equality, so $\mathcal T_E$ is not $T_0$. It is non-trivial: a permutation preserving $\mathcal T_E$ preserves $E=\,\equiv_{\mathcal T_E}$, and since $E$ is not $\operatorname{Sym}(\Omega)$-invariant, some permutation does not preserve $\mathcal T_E$.

(b) The specialization preorder $\le_{\mathcal T}$ of a $G$-invariant topology is $G$-invariant. If $G$ is strongly primitive it is equality ($\mathcal T$ is $T_1$) or universal ($\mathcal T$ indiscrete, hence trivial). Conversely, if $G$ preserves a non-trivial preorder $P$, then $\mathcal A(P)$ is $G$-invariant, its specialization preorder is $P\neq$ equality, so it is not $T_1$, and it is non-trivial because a permutation preserving $\mathcal A(P)$ preserves its specialization preorder $P$, which is not $\operatorname{Sym}(\Omega)$-invariant. The reformulation in terms of $\Sp$ and $\Cm$ is Theorem~\ref{thm:hierarchy}(b).
\end{proof}

We now show what primitivity forces on the whole hierarchy; the following theorem is new.

\begin{theorem}\label{thm:invariant}
Let $G$ be a transitive permutation group on $\Omega$, let $\mathcal T$ be a $G$-invariant topology, $X=(\Omega,\mathcal T)$, and let $\Gamma$ be one of $\Sp,\Ins,\Cm,\Ury,\Fun$.
\begin{enumerate}[label=\textup{(\alph*)}]
\item $\Gamma(X)$ is a $G$-invariant, vertex-transitive graph; in particular it is regular \textup{(}possibly of infinite valency\textup{)}.
\item If $G$ is primitive, then each $\Gamma(X)$ is either null or connected; $\Fun(X)$ is null or complete; and if $\Ins(X)$ is a disjoint union of cliques then it is null or complete.
\item If $G$ is primitive and $\mathcal T$ contains a non-empty finite open set, then $\mathcal T$ is discrete, or $\Omega$ is finite and $\mathcal T$ is indiscrete.
\item If $\Omega$ is finite, then $\Omega$ is partitioned into blocks of imprimitivity of equal size $b$, each an indiscrete open subspace, and $\mathcal T$ is the topology of unions of blocks. All five graphs coincide and are equal to the disjoint union of $|\Omega|/b$ copies of $K_b$. In particular, if $G$ is primitive, all five graphs are null \textup{(}$\mathcal T$ discrete\textup{)} or all are complete \textup{(}$\mathcal T$ indiscrete\textup{)}.
\item If $\mathcal T$ is Alexandrov, $T_0$ and not discrete, then $X$ has neither open points nor closed points, every minimal open set $U_x$ and every point closure $\cl\{x\}$ is infinite, and every vertex of $\Sp(X)$ has infinite valency. In particular $\Omega$ is infinite.
\end{enumerate}
\end{theorem}

\begin{proof}
(a) Every element of $G$ is a homeomorphism, hence an automorphism of $\Gamma(X)$ (Section~\ref{sec:hierarchy}). Since $G$ is transitive on $\Omega$, $\Gamma(X)$ is vertex-transitive, and vertex-transitive graphs are regular.

(b) The relation ``$x$ and $y$ lie in the same connected component of $\Gamma(X)$'' is an equivalence relation preserved by every automorphism of $\Gamma(X)$, in particular by $G$. By primitivity it is equality, in which case $\Gamma(X)$ is null, or universal, in which case $\Gamma(X)$ is connected. The relation $x\eqadj{\Fun}y$ is itself a $G$-invariant equivalence relation (Theorem~\ref{thm:hierarchy}(c)), so it is equality or universal, i.e.\ $\Fun(X)$ is null or complete. If $\Ins(X)$ is a disjoint union of cliques, then $\eqadj{\Ins}$ is a $G$-invariant equivalence relation and the same argument applies.

(c) Among the non-empty finite open sets choose one, $U$, of smallest cardinality. For $g\in G$ the set $gU$ is again a non-empty finite open set of the same size. If $gU\cap U\neq\emptyset$ then $gU\cap U$ is a non-empty open subset of $U$, so by minimality $gU\cap U=U$ and likewise $gU\cap U=gU$, whence $gU=U$. Thus the translates of $U$ are pairwise equal or disjoint, i.e.\ $U$ is a block of imprimitivity of the transitive group $G$. By primitivity $U$ is a singleton or $U=\Omega$. In the first case every point is open by transitivity, so $\mathcal T$ is discrete. In the second case $\Omega$ is finite and, being a minimal non-empty open set, it is the only one; so $\mathcal T$ is indiscrete.

(d) Let $\Omega$ be finite, so $X$ is Alexandrov. The indistinguishability classes are the fibres of the Kolmogorov quotient $q\colon X\to KX$; they are blocks of imprimitivity for $G$ (the relation is $G$-invariant), so they all have the same size $b$, and $G$ acts transitively by homeomorphisms on the finite $T_0$ space $KX$. For $c\in KX$ let $h(c)$ be the largest $r$ for which there is a chain $c_0<c_1<\dots<c_r=c$ in $KX$. Homeomorphisms preserve the order, hence $h$, so $h$ is constant on $KX$ by transitivity. If some $c$ were not minimal, say $c'<c$, then $h(c)\ge h(c')+1$, a contradiction. Hence $KX$ is an antichain, i.e.\ a finite discrete space. Consequently each fibre $q^{-1}(c)$ is open (as $\{c\}$ is open) and carries the indiscrete topology, and the open sets of $X$ are exactly the unions of fibres. All five graphs of a discrete space are null, so by Theorem~\ref{thm:hierarchy}(e) all five graphs of $X$ are obtained by replacing each point of $KX$ with a clique on its fibre; this is the disjoint union of $|\Omega|/b$ copies of $K_b$ in each case. If $G$ is primitive, the block system is trivial: $b=1$ (discrete) or $b=|\Omega|$ (indiscrete).

(e) Suppose $\mathcal T$ is Alexandrov and $T_0$, so that $(\Omega,\le)$ is a partial order and $\mathcal T=\mathcal A(\le)$. If some point were open, i.e.\ minimal, then by transitivity every point would be minimal, $\le$ would be equality and $\mathcal T$ would be discrete. Hence there are no minimal elements; dually (closed points are the maximal elements) there are no maximal elements. If some $U_x={\downarrow}x$ were finite, it would contain a minimal element of itself, which is minimal in $\Omega$ because $U_x$ is a down-set; this is impossible, so every $U_x$ is infinite, and dually every $\cl\{x\}={\uparrow}x$ is infinite. The neighbours of $x$ in $\Sp(X)$ include $U_x\setminus\{x\}$, so $x$ has infinite valency.
\end{proof}

Part (c) recovers the observation of~\cite[\S3]{CamV} that a minimal open set is a block of imprimitivity. For infinite $\Omega$ the following dichotomy, which is the heart of Cameron's simplified proof of the Macpherson--Praeger theorem, becomes a statement about the inseparability graph.

\begin{theorem}[~\cite{Cam-top,CamV}]\label{thm:dichotomy}
Let $G$ be a primitive permutation group on a countable set $\Omega$ and $\mathcal T$ a non-trivial, non-discrete $G$-invariant topology, $X=(\Omega,\mathcal T)$. Then exactly one of the following holds:
\begin{enumerate}[label=\textup{(\roman*)}]
\item $\Ins(X)$ is complete, i.e.\ $X$ is hyperconnected \textup{(}any two non-empty open sets meet\textup{)}; then the sets containing a non-empty open set form a non-trivial $G$-invariant filter;
\item $\Ins(X)$ has an infinite independent set \textup{(}the separation graph has an infinite clique\textup{)}; then $X$ has an infinite discrete subspace, and the complements of finite unions of discrete subsets form a non-trivial $G$-invariant filter.
\end{enumerate}
\end{theorem}

We refer to~\cite{Cam-top} for the proof. Combined with Theorem~\ref{thm:hierarchy}(f), for relational topologies it takes the following purely order-theoretic form.

\begin{corollary}\label{cor:dichotomy-alex}
Let $G$ be a primitive group of automorphisms of a countable partial order $(\Omega,\le)$ which is not an antichain. Then either any two elements of $\Omega$ have a common lower bound, or there is an infinite subset of $\Omega$ no two elements of which have a common lower bound. The same holds with ``upper'' in place of ``lower''.
\end{corollary}
\begin{proof}
The Alexandrov topology $\mathcal A(\le)$ is $G$-invariant, $T_0$ and non-discrete; it is non-trivial because a trivial non-discrete $T_0$ topology would be $\operatorname{Sym}(\Omega)$-invariant with a $\operatorname{Sym}(\Omega)$-invariant specialization order, which by Lemma~\ref{lem:trivial-rel} is equality, forcing discreteness. By Theorem~\ref{thm:hierarchy}(f), $x\sim_{\Ins}y$ iff $x,y$ have a common lower bound; now apply Theorem~\ref{thm:dichotomy}. For upper bounds apply this to the opposite order, which $G$ also preserves.
\end{proof}

\begin{example}[The generic poset]\label{ex:generic}
Let $P$ be the countable universal homogeneous partial order (the Fra\"iss\'e limit of the finite posets~\cite{CamV}), with its Alexandrov topology. $P$ has the \emph{extension property}: for disjoint finite sets $A,B,C\subseteq P$ with every element of $A$ below every element of $B$, there is $z$ with $A<z<B$ and $z$ incomparable to all of $C$ (provided this is consistent). We claim:
\begin{enumerate}[label=(\roman*)]
\item $\operatorname{Aut}(P)$ is primitive but not strongly primitive;
\item $\Ins(P)=\Cm(P)=\Ury(P)=\Fun(P)$ is complete, while $\Sp(P)$ is the comparability graph of $P$, which is neither complete nor null.
\end{enumerate}
By homogeneity $\operatorname{Aut}(P)$ is transitive on points and on ordered pairs of each of the three types $x<y$, $x>y$, $x\perp y$. Let $E$ be an invariant equivalence relation other than equality. If $E$ contains a comparable pair, it contains all comparable pairs; any $x,y$ have a common upper bound $z$ (extension property), so $x\,E\,z\,E\,y$ and $E$ is universal. If $E$ contains an incomparable pair, it contains all of them; for $x<y$ choose $z$ incomparable to both, so again $x\,E\,z\,E\,y$. Thus $\operatorname{Aut}(P)$ is primitive; it is not strongly primitive since it preserves $\le$. For (ii), any two points have a common lower bound and a common upper bound, so $\Ins(P)$ and $\Cm(P)$ are complete, hence so are $\Ury(P)$ and $\Fun(P)$; and $P$ has both comparable and incomparable pairs. This illustrates Theorem~\ref{thm:invariant}(e) (no open or closed points), and alternative (i) of Theorem~\ref{thm:dichotomy}.

The \emph{generic preorder} $\tilde P$, obtained by replacing each point of $P$ by a countably infinite set of indistinguishable points~\cite{CamV}, has $K\tilde P=P$; by Theorem~\ref{thm:hierarchy}(e) its graphs are obtained from those of $P$ by blowing up vertices into infinite cliques. Its automorphism group is imprimitive and its topology is not $T_0$, as Proposition~\ref{prop:primitive-T0}(a) predicts. By contrast, for the \emph{rational world}, $\mathbb Q$ with its usual topology~\cite{Neumann-rw}, all five graphs are null since $\mathbb Q$ is metrizable; by Mekler's theorem~\cite{Mekler}, every countable group acting on a countable set and satisfying Mekler's condition~\cite[Proposition~4.1(a)]{CamV} acts faithfully on a countable space all of whose five graphs are null.
\end{example}

\subsection{Neighbourhood filters}
For a graph $\Gamma$ on a countable set $\Omega$, the \emph{neighbourhood filter} $\mathcal F_\Gamma$ is the filter generated by the open neighbourhoods $\Gamma(v)=N(v)$, i.e.\ the family of sets containing a finite intersection of neighbourhoods~\cite{CamV}. It is \emph{non-trivial} if it does not contain $\emptyset$, i.e.\ if any finitely many vertices have a common neighbour. Let $R$ denote the Rado graph~\cite{Cam-random}. Cameron proved~\cite[Theorem~5.2]{CamV} that for a countable graph $\Gamma$ the following are equivalent: $\mathcal F_\Gamma$ is non-trivial; $\Gamma$ contains $R$ as a spanning subgraph; $\mathcal F_\Gamma\subseteq\mathcal F_R$ for a suitable copy of $R$ on $\Omega$. We apply this to graphs on countable spaces.

\begin{theorem}\label{thm:filters}
Let $X$ be a countably infinite topological space.
\begin{enumerate}[label=\textup{(\alph*)}]
\item The filters are reverse-ordered along the hierarchy: $\mathcal F_{\Fun(X)}\subseteq\mathcal F_{\Ury(X)}\subseteq\mathcal F_{\Ins(X)}\cap\mathcal F_{\Cm(X)}$ and $\mathcal F_{\Ins(X)}\cup\mathcal F_{\Cm(X)}\subseteq\mathcal F_{\Sp(X)}$. If one of them is non-trivial, so are all those above it in the hierarchy; in particular $\Fun(X)$, $\Ury(X)$ contain the Rado graph as a spanning subgraph whenever $\Ins(X)$ or $\Cm(X)$ does.
\item $\mathcal F_{\Fun(X)}$ is non-trivial if and only if every continuous function $X\to\R$ is constant; in that case $\mathcal F_{\Fun(X)}$ is the cofinite filter.
\item If $\mathcal F_{\Gamma(X)}$ is non-trivial for some $\Gamma$ in the hierarchy, then every continuous real function on $X$ is constant; in particular $X$ is connected, and $X$ is not $T_1$ if $\Gamma\in\{\Sp,\Cm\}$, not Hausdorff if $\Gamma=\Ins$ and not Urysohn if $\Gamma=\Ury$.
\item If $X$ is hyperconnected then $\mathcal F_{\Ins(X)}$, $\mathcal F_{\Ury(X)}$ and $\mathcal F_{\Fun(X)}$ all equal the cofinite filter.
\item Let $X$ be Alexandrov and $T_0$, i.e.\ a countable poset. If $X$ is downward directed and has no minimal elements, then $\mathcal F_{\Sp(X)}$ is non-trivial, so the comparability graph $\Sp(X)$ contains the Rado graph as a spanning subgraph, while $\Ins(X)$ is complete. If moreover some element of $X$ is incomparable to infinitely many others, then $\mathcal F_{\Ins(X)}\subsetneq\mathcal F_{\Sp(X)}$.
\end{enumerate}
\end{theorem}

\begin{proof}
(a) If $E(\Gamma)\subseteq E(\Gamma')$ then $\Gamma(v)\subseteq\Gamma'(v)$ for every $v$. Every generator $\Gamma'(v)$ of $\mathcal F_{\Gamma'}$ therefore contains the generator $\Gamma(v)$ of $\mathcal F_{\Gamma}$, so $\mathcal F_{\Gamma'}\subseteq\mathcal F_{\Gamma}$; and if every finite intersection $\Gamma(v_1)\cap\dots\cap\Gamma(v_n)$ is non-empty, so is $\Gamma'(v_1)\cap\dots\cap\Gamma'(v_n)$. Now apply Theorem~\ref{thm:hierarchy}(a), and Cameron's theorem for the statement about the Rado graph.

(b) By Theorem~\ref{thm:hierarchy}(c), $\Fun(X)$ is a disjoint union of cliques. If it has two classes $C\ne C'$, then vertices $v\in C$, $v'\in C'$ have no common neighbour, so $\emptyset=\Fun(v)\cap\Fun(v')\in\mathcal F_{\Fun}$. If there is only one class, then $\Fun(X)$ is complete (all continuous real functions are constant), $\Fun(v)=\Omega\setminus\{v\}$, and the finite intersections of these sets are exactly the cofinite sets; as $\Omega$ is infinite they are non-empty, and the generated filter is the cofinite filter.

(c) By (a), non-triviality of $\mathcal F_{\Gamma(X)}$ implies non-triviality of $\mathcal F_{\Fun(X)}$, hence by (b) that every continuous real function is constant. If $X=A\sqcup B$ with $A,B$ non-empty and open, the characteristic function of $A$ is continuous and non-constant; so $X$ is connected. A non-trivial neighbourhood filter forces the graph to have edges (indeed every vertex has a neighbour), and Theorem~\ref{thm:hierarchy}(b) gives the remaining statements.

(d) If any two non-empty open sets meet, then any two points are inseparable, so $\Ins(X)$ is complete; by Theorem~\ref{thm:hierarchy}(a) so are $\Ury(X)$ and $\Fun(X)$, and the filter generated by the sets $\Omega\setminus\{v\}$ is the cofinite filter, as in (b).

(e) Let $A\subseteq X$ be finite and non-empty. By downward directedness $A$ has a lower bound $b$. Since $X$ has no minimal elements, there is an infinite strictly descending chain $b>b_1>b_2>\cdots$; as $A$ is finite, some $b_i\notin A$. Then $b_i<a$ for all $a\in A$, so $b_i$ is a neighbour of every $a\in A$ in $\Sp(X)$. Hence finite intersections of neighbourhoods in $\Sp(X)$ are non-empty, $\mathcal F_{\Sp(X)}$ is non-trivial, and by Cameron's theorem $\Sp(X)$ contains $R$ as a spanning subgraph. Any two points have a common lower bound, so $\Ins(X)$ is complete (Theorem~\ref{thm:hierarchy}(f)) and $\mathcal F_{\Ins(X)}$ is the cofinite filter by (d). Finally, if $v$ is incomparable to infinitely many elements, then $\Sp(v)$ is not cofinite, so $\Sp(v)\in\mathcal F_{\Sp(X)}\setminus\mathcal F_{\Ins(X)}$.
\end{proof}

\begin{example}
The generic poset $P$ of Example~\ref{ex:generic} satisfies the hypotheses of Theo\-rem~\ref{thm:filters}(e). Therefore its comparability graph $\Sp(P)$, all of whose finite induced subgraphs are perfect (they are comparability graphs), contains the Rado graph (which is universal, hence far from perfect) as a spanning subgraph; and
\[
\mathcal F_{\Fun(P)}=\mathcal F_{\Ury(P)}=\mathcal F_{\Ins(P)}=\mathcal F_{\Cm(P)}=\text{cofinite filter}\subsetneq\mathcal F_{\Sp(P)} .
\]
This is a topological instance of Cameron's observation~\cite[\S5]{CamV} that nested spanning subgraphs containing $R$ give nested neighbourhood filters. The integer lattice $\mathbb Z^n$ ($n\ge2$) with the product order is another example.
\end{example}

\begin{remark}\label{rem:countable-XGamma}
Replacing each ray $(0,1)$ in Definition~\ref{def:XGamma} by $(0,1)\cap\mathbb Q$ yields a \emph{countable} $T_1$ space $X^{\mathbb Q}_\Gamma$ for which Theorem~\ref{thm:XGamma}(a)--(d) hold with the same proofs. Taking $\Gamma=R$ gives a countable space whose inseparability graph is the Rado graph together with countably many isolated vertices.Because of the isolated vertices its neighbourhood filter is trivial. Whether the isolated vertices can be removed, that is, whether some countable topological space has inseparability graph isomorphic to the Rado graph, remains open.
\end{remark}

\section{Computational data}\label{sec:data}

We enumerated all finite $T_0$ spaces with at most $7$ points up to homeomorphism (equivalently, all posets up to isomorphism; there are $1,2,5,16,63,318,2045$ of them) and computed the five graphs. The identity $\Ury(X)=\Ins(X)\sq$, all statements of Theorems~\ref{thm:realization}, \ref{thm:equality}, \ref{thm:homotopy}(a),(b) (the latter at the level of rational homology) and Propositions~\ref{prop:alpha}, \ref{prop:distances} were verified by computer for all these spaces, and all graphs on at most $7$ vertices were tested against the criterion of Theorem~\ref{thm:realization}. Some of the resulting counts are shown in Table~\ref{tab:data}. 

\begin{table}[ht]
\centering
\begin{tabular}{@{}lrrrrrrr@{}}
\toprule
$n$ & 1 & 2 & 3 & 4 & 5 & 6 & 7\\
\midrule
$T_0$ spaces on $n$ points & 1 & 2 & 5 & 16 & 63 & 318 & 2045\\
graphs on $n$ vertices & 1 & 2 & 4 & 11 & 34 & 156 & 1044\\
graphs $\cong\Ins(X)$, $|X|=n$ & 1 & 2 & 4 & 9 & 21 & 56 & 164\\
spaces with $\Sp=\Ins$ & 1 & 2 & 4 & 9 & 20 & 48 & 115\\
spaces with $\Ins=\Ury\,(=\Fun)$ & 1 & 2 & 4 & 10 & 28 & 100 & 446\\
spaces with $\Ins(X)$ a cograph & 1 & 2 & 5 & 16 & 62 & 304 & 1859\\
spaces with $\Ins(X)$ chordal & 1 & 2 & 5 & 16 & 63 & 318 & 2045\\
\bottomrule
\end{tabular}
\caption{Counts for finite $T_0$ spaces with $n\le7$ points. The third row counts isomorphism classes of graphs on $n$ vertices which are inseparability graphs of $T_0$ spaces with $n$ points; by Theorem~\ref{thm:realization} these are exactly the simplicially covered graphs. The last row illustrates Proposition~\ref{prop:holes}.}
\label{tab:data}
\end{table}

The smallest spaces whose inseparability graph is not a cograph have $5$ points, and among the $63$ five-point $T_0$ spaces exactly one fails: the ``$W$-shaped'' space with open points $m_1,m_2,m_3$ and closed points $x_1>m_1,m_2$ and $x_2>m_2,m_3$. Here $M(x_1)=\{m_1,m_2\}$ and $M(x_2)=\{m_2,m_3\}$, so $m_1\sim x_1\sim x_2\sim m_3$ in $\Ins$, while $m_1\not\sim x_2$, $x_1\not\sim m_3$ and $m_1\not\sim m_3$: the four points induce $P_4$. By contrast, Proposition~\ref{prop:holes} shows that $\Ins(X)$ is chordal for all spaces with at most $7$ points, as confirmed by the last row of the table.\\
The questions studied in this paper were inspired by the programme on graphs defined on algebraic structures developed in~\cite{Cam-groups,ACNS,Cam-AKCE,CDD}, and Section~\ref{sec:groups} by Cameron's work on topologies, filters and the Rado graph~\cite{CamV,Cam-top}.

\subsection*{Acknowledgement}
The second author thanks the management of Muthoot Institute of Technology and Science, Kochi, Kerala, India for providing the facilities to conduct the research.

\end{document}